\documentclass[5p]{elsarticle}
\usepackage{amsmath,amssymb,amsfonts}
\usepackage{graphicx,color}
\usepackage{color}
\usepackage[table]{xcolor}
\usepackage{subcaption}

\newtheorem{theorem}{Theorem}[section]
\newtheorem{lemma}[theorem]{Lemma}
\newtheorem{definition}{Definition}
\newtheorem{corollary}[theorem]{Corollary}
\newtheorem{proposition}[theorem]{Proposition}

\newtheorem{remark}{Remark}
\newtheorem{assumption}[theorem]{Assumption}
\newtheorem{example}{Example}

\newcommand{\transpose}{\mathsf{T}} %or \top or \intercal

\newcommand*{\QEDB}{\hfill\ensuremath{\square}}
\newcommand{\ol}{\overline}
\newcommand{\bs}{\boldsymbol}

\newcommand{\ta}{\widetilde{a}}
\newcommand{\tb}{\widetilde{b}}
\newcommand{\tc}{\widetilde{c}}
\newcommand{\mbP}{\mathbb{P}}

\newcommand{\calC}{\mathcal{C}}

\newcommand{\calP}{\mathcal{P}}
\newcommand{\calS}{\mathcal{S}}
\newcommand{\mFJ}{\mathcal{F}_J}

\newcommand{\trans}{\mathsf{T}} %or \top or \intercal

\DeclareSymbolFont{bbold}{U}{bbold}{m}{n}
\DeclareSymbolFontAlphabet{\mathbbold}{bbold}

\newcommand\oprocendsymbol{\hbox{$\square$}}
\newcommand\oprocend{\relax\ifmmode\else\unskip\hfill\fi\oprocendsymbol}

\newcommand*{\QEDA}{\hfill\ensuremath{\blacksquare}}%

\newcommand{\QED}{\hfill \mbox{\raggedright \rule{.1in}{.1in}}}
\newenvironment{proof}{\vspace{1ex}\noindent{\itshape Proof:}\hspace{0.5em}}
{\hfill\QED\vspace{1ex}}

\graphicspath{{images/}}

\begin{document}
\begin{frontmatter}
  \title{{\bf \LARGE Network Theoretic Analysis of Maximum a Posteriori Detectors for Optimal Input Detection}\tnoteref{t1}}\tnotetext[t1]{This material is based upon
    work supported in part by ARO award 71603NSYIP and in part by
    UCOP award LFR-18-548175.}

  % The authors
  % are with the Department of Mechanical Engineering, University of
  % California at Riverside.
  
 \author[Riverside]{Rajasekhar Anguluri}\ead{ranguluri@engr.ucr.edu}  
  \author[Riverside]{Vaibhav Katewa}\ead{vkatewa@engr.ucr.edu} 
  \author[Pullman]{Sandip Roy}\ead{sroy@eecs.wsu.edu}  
  \author[Riverside]{Fabio Pasqualetti}\ead{fabiopas@engr.ucr.edu}  
%  \address{Department of Mechanical Engineering, University of California, Riverside, CA, USA}
 % \ead{\{ranguluri,vkatewa, fabiopas\}@engr.ucr.edu}

%\author[Pullman]{Sandip Roy}\ead{sroy@eecs.wsu.edu}               % e-mail address 

\address[Riverside]{Department of Mechanical Engineering, University of California, Riverside, CA, USA}  % Please supply                                              
\address[Pullman]{School of Electrical Engineering and Computer Science, Washington State University, Pullman, WA, USA}     
%  \thanks{This material is based upon work supported in part by ARO
%    award 71603NSYIP, and in part by UCOP award LFR-18-548175. The
%    authors are with the Department of Mechanical Engineering,
%    University of California at Riverside.
%    \href{mailto:vkatewa@engr.ucr.edu}{\texttt{\{vkatewa,}}
%    \href{mailto:rangu003@engr.ucr.edu}{\texttt{rangu003,}}
%    \href{mailto:fabiopas@engr.ucr.edu}{\texttt{fabiopas\}}}
%    \texttt{@engr.ucr.edu}.}}

%\maketitle
%This paper studies a detection problem for network systems, where changes in the statistical properties of
%an input driving certain network nodes has to be detected by
%sparse and remotely located sensors.

\begin{abstract} 

 This paper considers maximum-a-posteriori (MAP) and linear discriminant based MAP detectors to detect changes in the mean and covariance of a stochastic input, driving specific network nodes, using noisy measurements from sensors non-collocated with the input nodes. 
We explicitly characterize both detectors' performance in terms of the network edge weights and input and sensor nodes' location. In the asymptotic measurement regime, when the input and measurement noise are jointly Gaussian, we show that the detectors' performance can be studied using the input to output gain of the system's transfer function matrix. Using this result, we obtain conditions for which the detection performance associated with the sensors on a given network cut is better (or worse) than that of the sensors associated with the subnetwork
induced by the cut and not containing the input nodes. Our results also provide structural insights into the sensor placement from a detection-theoretic viewpoint. We validate our theoretical findings via multiple numerical examples.
 \end{abstract}

\end{frontmatter}

\section{Introduction}

{ Security of cyber-physical networks is of timely and
  utmost importance \cite{FP-FD-FB:10y}. In recent years, researchers
  have proposed a number of model-based and heuristic approaches for
  detecting and mitigating attacks against the actuators and the
  sensors in the network (see \cite{HSS-DR-TE-VP-JQ:19} and the
  references therein). Despite the success of these studies in
  revealing the performance and the limitations of attack detection
  mechanisms, several challenges remain, particularly in
  distinguishing malicious signals from ambient data, selecting
  optimal sensor locations to maximize the detection performance
  \cite{SR-JAT-MX:16, LY-SR-SS:18}, and deriving simple graphical
  rubrics to readily evaluate and optimize network security
  \cite{RD-JAT-SR:15, RA-RD-SR-FP:16}.

  % First, due to the limited budget and accessibility, often only a few
  % sensing devices can be deployed in a network, and thus simple
  % graphical rubrics/metrics rather than precise algorithms are needed
  % for detecting malicious behavior in network systems
  % \cite{RD-JAT-SR:15, RA-RD-SR-FP:16}. Additionally, methods for
  % resource-constrained sensor placement need to be developed for
  % detection purposes, similar to the studies conducted in
  % \cite{SR-JAT-MX:16, LY-SR-SS:18}.

  This study contributes to a growing research effort on
  characterizing dynamic properties of network systems, including
  observability, estimation, detection, and input-output behaviors
  \cite{SR-JAT-MX:16, SR-MX-SS:18, LY-SJ-BA:13, AV-CJ-MX-SR-SW:19}. In
  particular, we provide network theoretic insights into the detection
  of changes in the statistical properties of a stationary stochastic
  input driving certain network nodes, with the ultimate objective of
  informing the placement of sensors for detection.

  % The basic idea is that the network topological structure enforces
  % relationships between signals at different locations in the network,
  % so that changes to the dynamics at one location can potentially be
  % detected using measurement signatures at other locations in the
  % network.  Thus, local changes implicate a predictable propagative
  % response across the network, which can in theory be used to identify
  % the change from sparse and remote measurements.

  Our work is also associated with the recent studies on constrained
  sensor selection and actuator placement in network systems
  \cite{THS-FLC-JL:16, HZ-RA-SS:17, JAT-SR-YW:17}. It is also aligned
  with network-theoretic studies that consider metrics for detection
  and estimation
  \cite{SR-MX-SS:18,THS-FLC-JL:16,AV-CJ-MX-SR-SW:19,FP-SZ-FB:14,SZ-FP:17,HJV-MK-HL:17}. Compared
  to these works, we pursue an explicit characterization of the
  relationships between the detection performance of a set of
  sensors and the graphical structure of the system.

  % Loosely speaking, these works have established that optimal sensor
  % placement or selection in networks is generally a computationally
  % difficult problem, for which suboptimal solutions can be obtained
  % using heuristic algorithms in some interesting scenarios. Relative
  % to this existing line of work, our studies are motivated by the
  % simple idea that specially-structured network dynamics may sometimes
  % admit structural insights into detector performance which informs
  % satisfactory (although perhaps suboptimal) sensor placement. The
  % following are the main contributions of this work:

\smallskip 
\noindent
\textbf{Contributions:}\footnote {In a preliminary version of this
  paper \cite{RA-RD-SR-FP:16}, we considered a SISO system, and
  studied the MAP detector's performance for changes in mean of the
  stochastic input, assuming noiseless measurements. Instead, in this
  paper, we consider a MIMO system, and study the detector's
  performance for changes occurring in both mean and covariance, under
  noisy and noiseless measurements. In addition, this paper also
  includes results on networks with non-negative edge weights.}  The
main contributions of this work are as follows. First, we consider a
binary hypothesis testing problem for a discrete time Gaussian process
driving the linear network dynamics through certain network nodes. We
primarily consider the scenario where hypothesis on either the mean or
the covariance of the input process must be detected using the
measurements (possibly corrupted with white Gaussian noise) collected
from output nodes that are at least at a specified distance apart from
the input nodes. We characterize the \textit{maximum a posteriori}
(MAP) detector, and quantify its performance as a function of the gain
of the input-output transfer matrix of the network system. These
results are significant in their own rights. For instance, besides our
results, there are only limited works related to detecting changes in
the covariance of unknown input signals, and this problem is highly
relevant in the context of cyber-physical security.

Second, we study the MAP detector's performance as a function of the
sensors location. In the absence of noise, regardless of the network
structure and edge weights, we show that the performance of the
detector associated with a set of sensors forming a cut of the network
(nodes on the cut shall be referred as to cutset nodes) is as good as
the performance obtained by measuring all nodes of the subnetwork
identified by the cut and not containing the nodes affected by the
input nodes (referred as partitioned set nodes). Instead, in the
presence of noise, depending upon the transfer matrix gain between the
cutset nodes and the partitioned nodes, we show that the detection
performance of sensors on the cutset nodes may be better or worse than
those of sensors on the partitioned nodes. Finally, we demonstrate our
theoretical findings on Toeplitz line networks and some illustrative
numerical examples.

Our analysis leads to the interesting results that, depending on the network weights and structure, and the intensity of sensor noise, the detection performance may improve as the graphical distance between the input nodes and the sensors location increases. In fact, our results (i) inform the optimal positioning of sensors for the detection of failure of system components or malicious tampering modeled by unknown stochastic inputs, (ii) allow for the detection of unexpected changes of the system structure, because such changes would modify the original detection profile, and (iii) provide network design guidelines to facilitate or prevent measurability of certain network signals.}

\smallskip 
\noindent
\textbf{Mathematical notation:} The cardinality of a set $A$ is denoted by $\text{card}(A)$. The set of natural numbers, real numbers, and complex numbers are denoted as $\mathbb{N}$, $\mathbb{R}$, and $\mathbb{C}$, respectively. The eigenspectrum and the spectral radius of a matrix $M \in \mathbb{C}^{n \times n}$ are denoted by $\text{spec}(M)$ and $\overline{\lambda}(M)$, resp. A symmetric positive (resp. semi) definite matrix $M$ is denoted as $M\succ0$ (resp. $M\succeq0$). Instead, a non-negative matrix $M$ is denoted as $M \geq 0$. Let $M_1,\ldots,M_n$ be matrices of different dimensions, then $\mathrm{diag}(M_1, \cdots,M_N)$ represents a block diagonal matrix. The Kronecker product of $M_1$ and $M_2$ is denoted by $M_1\otimes M_2$. An $n\times n$ identity matrix is denoted by $I$ or $I_n$. The norm on the Banach space of matrix-valued functions, that are essentially bounded on the unit circle $\{z \in {\mathbb C}: |z|=1\}$, is defined as $\|F(z)\|_{\infty}:=\text{ess sup}\|F(z)\|_2$ \cite{GI-KMA-GS:93}. All finite dimensional vectors are denoted by bold faced symbols. The set $\{{\bf e}_1,\ldots, {\bf e}_n\}$ denotes the standard basis vectors of $\mathbb{R}^n$. Let ${\bf x}=(x_1,\ldots,x_n)^\transpose$ and ${\bf y}=(y_1,\ldots,y_n)^\transpose$. Then, ${\bf x}\leq {\bf y}$ if and only if $x_i\leq y_i$, $i=1,\ldots,n$. The probability of an event $\mathcal{E}$ is denoted by $\mathrm{Pr}[\mathcal{E}]$. The conditional mean and covariance of a random variable (vector) $X$ is denoted by $\mathbb{E}[X|\mathcal{E}]$ and $\mathrm{Cov}[X|\mathcal{E}]$, respectively. For $Z\sim\mathcal{N}(0,1)$, $Q_{\mathcal{N}}(\tau)$ denotes $\text{Pr}\left[Z\geq \tau\right]$. For $Y \sim \chi^2(p)$, central chi-squared distribution with $q$ degrees of freedom, $Q_{\chi^2}(p,\tau)$ denotes $\text{Pr}\left[Y\geq \tau\right]$. 

\section{Preliminaries and problem setup}\label{sec: problem setup and preiliminary notions}
Consider a network
represented by the digraph
$\mathcal{G}:=(\mathcal{V}, \mathcal{E})$, where
$\mathcal{V}:= \{1,\ldots,n\}$ and
$\mathcal{E}\subseteq \mathcal{V}\times \mathcal{V}$ are the node and edge sets. Let $g_{ij}\in \mathbb{R}$ be the
weight assigned to the edge $(i,j)\in \mathcal{E}$, and
define the \textit{weighted adjacency matrix} of $\mathcal{G}$ as
$G:=[g_{ij}]$, where $g_{ij}=0$ whenever $(i,j)\notin \mathcal{E}$. 
Let $\mathcal{K}:= \{k_1,\ldots,k_r\}\subseteq \mathcal{V}$ be the set of input nodes, which receive $r$ inputs. Let $w(i,j)$ denote a path on $\mathcal{G}$ from node $i$ to $j$, and let $|w(i,j)|$ be the number of edges of $w(i,j)$. Define the distance between input node set $\cal {K}$ and a set of nodes $\cal {S}\subseteq \cal{V}$ as $\text{dist}({\cal K}, {\cal S}):=\text{min}\{|w(i,j)|: i\in {\cal K}, j\in{\cal S}\}$. 

We associate to each node $i$ a state $x_i \in \mathbb{R}$, and let the network evolve with discrete linear dynamics
\begin{align}\label{eq: system}
{\bf x}[k+1]=G{\bf x}[k]+\Pi {\bf w}[k], 
\end{align}
where ${\bf x}=[ x_1 \cdots x_n]^T \in \mathbb{R}^n$ contains the states of the nodes at time $k \in \mathbb{N}$, ${\bf x}[0]\sim \mathcal{N}({\bf 0},\Sigma_0)$ is the initial state, and ${\bf w}[k] \in \mathbb{R}^{r}$ is the input vector. The input matrix $\Pi=[{\bf e}_{k_1}, \ldots, {\bf e}_{k_r}]$ indicates the location of the input nodes. The input ${\bf w}[k]$ be governed by one of the following two competing statistical hypotheses: 
\begin{align}\label{eq: input hypothesis}
\begin{split}
H_1: \quad & {\bf w}[k]\overset{\text{i.i.d}}{\sim} \mathcal{N}\left({\boldsymbol \mu_1},{ \Sigma_1}\right), \quad  k=0,1,\ldots,N,\\
H_2:  \quad & {\bf w}[k]\overset{\text{i.i.d}}{\sim} \mathcal{N}\left({\boldsymbol \mu_2},{ \Sigma_2}\right), \quad  k=0,1,\ldots,N,
\end{split}
\end{align}
where the moments $\boldsymbol \mu_i \in \mathbb{R}^r$ and $\Sigma_i \in \mathbb{R}^{r \times r}(\succ 0)$, $i\in \{1,2\}$, are completely known. In other words, the competing hypotheses are simple. However, the true hypothesis is assumed to be unknown over the interval $k=0,1,\ldots,N$. We are concerned with detecting the true hypothesis on the input signal, using measurements from the sensors that are not collocated with the input nodes. 

We assume that the nodes $\mathcal{J}:= \{j_1,\ldots,j_m\}\subseteq \mathcal{V}$ are accessible for sensor placement (one sensor for each node), if $\text{dist}({\cal K}, {\cal J})\geq d$, where $d \in \mathbb{N}$. We refer to $\cal J$ as the sensor set. The output of these sensors is given by 
\begin{align}\label{eq: measurement model}
{\bf y}_\mathcal{J}[k] & = C{\bf x}[k]+{\bf v}[k], 
\end{align}
where  $C=[{\bf e}_{j_1}, \ldots, {\bf e}_{j_m}]^\transpose$ and ${\bf v}[k]\sim \mathcal{N}({\bf 0},\sigma^2_vI)$. Let the process $\{{\bf x}[0], {\bf w}[0], {\bf v}[0], {\bf w}[1], {\bf v}[1], \ldots\}$ be uncorrelated. To detect the true hypothesis, we task sensors with a detector, which maps the following time aggregated measurements
\begin{align}\label{eq: measurements recorded}
{\bf Y}^\transpose_\mathcal{J}=\begin{bmatrix}
{\bf y}^\transpose_\mathcal{J}[1] & {\bf y}^\transpose_\mathcal{J}[2]& \cdots {\bf y}^\transpose_\mathcal{J}[N]
\end{bmatrix}, 
\end{align}
to a detected hypothesis $\widehat H$. We will consider the \textit{maximum a posteriori probability} (MAP) detector, which is given by the following decision rule:
\begin{align}\label{eq: MAP}
\text{Pr}(\{H_{2} \, \text{is true}\}|{\bf Y}_{\cal J})&\overset{\widehat H=H_{2}}{\underset{\widehat H= H_{1}}{\gtrless}}\text{Pr}(\{H_{1} \, \text{is true}\}|{\bf Y}_{\cal J}). 
\end{align}

For a predetermined set of input nodes $\cal K$, the focus of our analysis is to characterize the performance of the detector \eqref{eq: MAP}, in terms of the network's adjacency matrix $G$. The performance of the detector \eqref{eq: MAP} is measured by its error probability, which is given by 
\begin{align}\label{eq: general error probability}
\mathbb{P}_e(\cal J)& = \sum_{i \in \{1,2\}}\text{Pr}(\widehat{H}\ne H_i| \{H_{i} \, \text{is true}\})\pi_i. 
\end{align}
where $\pi_i=\text{Pr}(\{H_i \text{ is true}\})$ is the prior probability. 

For any sensor set $\cal J$ that satisfies $\text{dist}({\cal K}, {\cal J})\geq d$, one expects that the MAP detector's performance \eqref{eq: general error probability} is maximum when $\text{dist}({\cal K}, {\cal J})=d$. However, for certain network configurations, studies have shown that the gain of transfer function, which is closely related to the \textit{signal-to-noise} ratio (SNR) of a detector, is maximum when the input and output nodes are farther apart \cite{JA-SR:15}. Hence, it remains unclear whether the closeness of the sensors to the input nodes improves the performance of the detector. 

In this paper, we show that the graphical proximity indeed modulate the MAP detector's performance, for certain classes of the detection problems \eqref{eq: input hypothesis}. In particular, we characterize networks for which the detection performance obtained when sensors are located on a \textit{node cutset} is better (or worse) than the performance obtained when sensors are placed on nodes of the subnetwork induced by the \textit{node cutset} that does not contain the input nodes (precise statements are provided in Section \ref{sec: network_analysis}). See Fig \ref{fig: cutset} for an illustration of \textit{node cutset} and the subnetwork (partitioned nodes) induced by it.  

Throughout the paper, we will distinguish the performance of sensors with measurement noise $\sigma^2_v=0$ and without noise $\sigma^2_v>0$. The essential reason, as we will show later, is that only in the presence of noise, we have a scenario where the performance of a MAP detector associated with the cutset nodes may be worse than that of the partitioned nodes. We end this section with an example that shows the impact of noise on the detection performance. 

\begin{figure}
	\centering
	\includegraphics[width=0.9\linewidth]{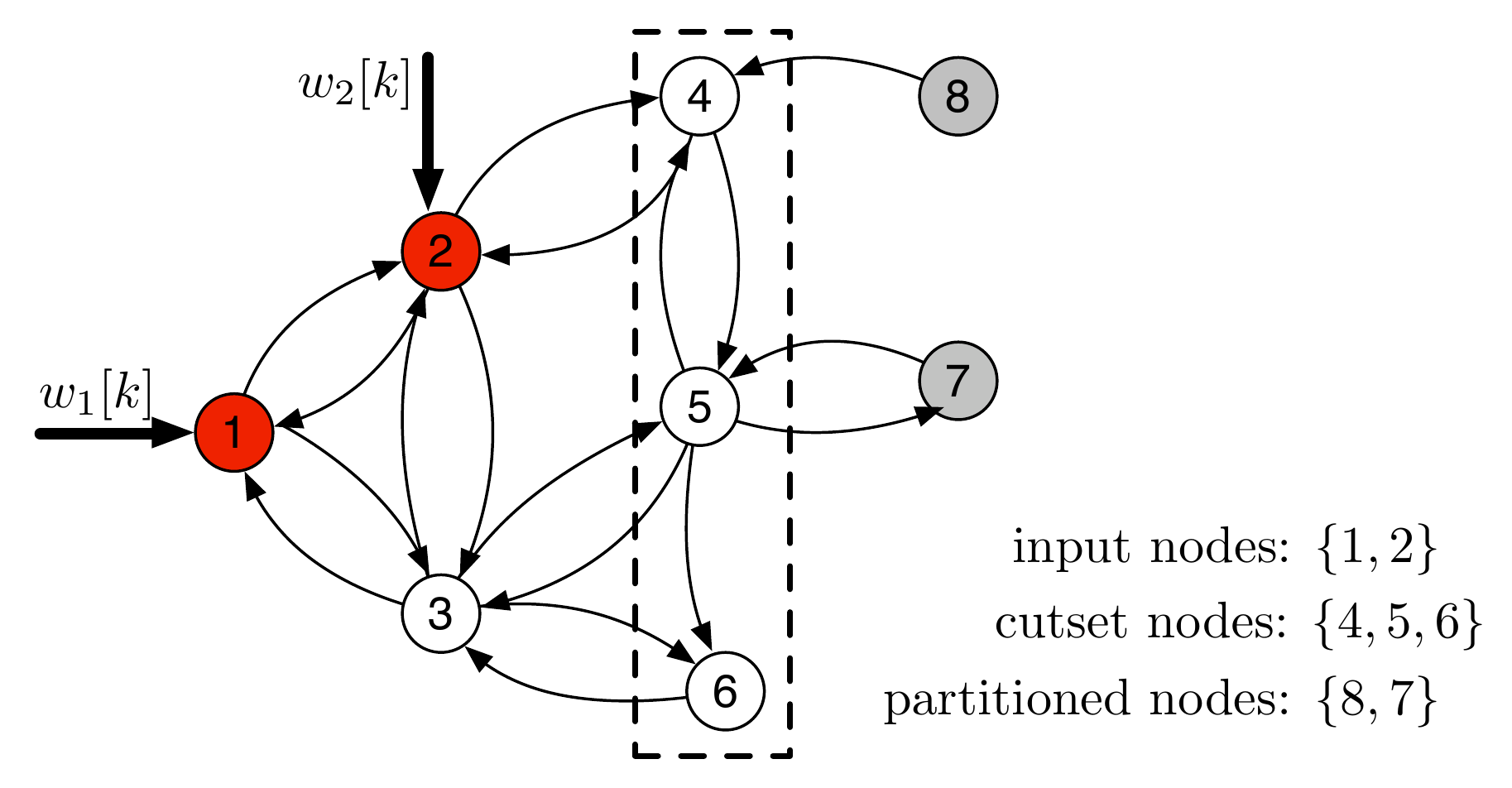}
	\caption{\footnotesize  Illustration of network partitions induced by a node cutset.}
	\label{fig: cutset}
\end{figure}

\begin{example}\label{exam: classical detection setup}
	Let $H_i: w[k]\sim \mathcal{N}(\mu_i, \sigma^2)$ and $y[k]=\alpha w[k]\!+\!v[k]$, where $\alpha$ is the tuning parameter and $v(k)\!\sim \!\mathcal{N}(0,\sigma^2_v)$. For $\pi_i=0.5$, the MAP detector's error probability, based on $N$ samples of $y$, is $0.5\,Q_{\cal N}(0.5 \,\eta)$. Here $\eta^2\!=\!N(\mu_1\!-\!\mu_2)^2/(\sigma^2+\alpha^{-2}\sigma^2_v)$ 
	denotes the SNR, and $Q_{\cal N}(0.5 \,\eta)$ is decreasing in $\eta$. For $\sigma^2_v=0$, the error probability does not depend on $\alpha$, and is lesser than the case where $\sigma^2>0$. However, when $\sigma_v^2>0$, $Q_{\cal N}(\cdot)$ depends on $\alpha$. In particular, when $\alpha$ is small, $Q_{\cal N}(\cdot)$ is high, and vice versa. Thus, in the presence of noise, the error probability can be reduced by properly tuning $\alpha$. \oprocend
\end{example}

%\begin{remark}{\bf \textit{(Non-testable hypotheses and zero dynamics)}} The hypothesis testing framework considered in this paper may not applicable to the situations such as the system \eqref{eq: system} associated with the measurement model  \eqref{eq: measurement model} has zero dynamics, i.e., ${\bf Y}_{\cal J}[k]={\bf 0}$ for non-zero input, or, some components of the input gets canceled in the operation of $C{\bf}{\bf x}[k]$. However, in the latter case, our framework can modified to analyze the input signals that lies in the low dimensional subspace. \oprocend
%\end{remark}

\section{Detection performance of the MAP detector}\label{sec:
  detector} 
{In this section, we derive the algebraic expressions of
  the MAP decision rules and their error probabilities for two special
  cases of the hypotheses in \eqref{eq: input hypothesis}. The first
  case is the \textit{mean shift model}, in which the covariance
  matrices in \eqref{eq: input hypothesis} are equal, but the mean
  vectors are different. The MAP detector (see \eqref{eq: MAP mean
    detector} for the decision rule) for this case is the optimal
  detector in the (Bayesian) sense that, for the simple hypotheses
  $H_1$ and $H_2$ defined as in \eqref{eq: input hypothesis}, no other
  detector outperforms the detection performance of \eqref{eq: MAP
    mean detector}. The second case is the \textit{covariance shift
    model} in which the mean vectors in \eqref{eq: input hypothesis}
  are equal, but the covariance matrices are different. For this
  latter case, we will rely on the sub optimal LD-MAP detector (see
  below) for deciding the hypothesis. The reason for working with
  these models is twofold: (i) the error probability expressions are
  analytically tractable and (ii) these models are widely used for
  detection and classification problems that arise in practice
  \cite{HLV:04, LLS:91}. The probability expressions derived in this
  section will be used for the network analysis of the MAP detector's
  performance (Section \ref{sec: network_analysis}). Our work draws on
  the extensive literature on hypothesis testing using Multivariate
  Gaussian data, but uses a specific simplified detector which allows
  development of network-theoretic results. Extension of our framework
  and corresponding results to more general setting is mentioned in
  Remarks \ref{rmk: other covariance structures} and~\ref{rmk:
    non-Gaussian}.

  The results mentioned in the section are obtained from a standard
  application of hypothesis testing tools for linear models. Yet, to
  the best of our knowledge, the asymptotic characterization of the
  error probability of the MAP detector (Lemma \ref{lma: special cases
    error probability}) is novel, and it serves as the starting point
  for the results presented in the subsequent sections.}

\smallskip 
\begin{definition}{\bf \textit{(Linear discriminant function-based MAP detector: LD-MAP)}}\label{def: LD-MAP} A LD-MAP detector is as in \eqref{eq: MAP} with ${\bf Y}_{\cal J}$ \eqref{eq: measurements recorded} replaced by the discriminant function $y={\bf b}^\transpose{\bf Y}_{\cal J}$, where the vector \footnote{In the literature of pattern recognition and communications, ${\bf b}$ is commonly referred as to the Fisher's discriminant and optimal SINR beam former, respectively  \cite{RD-PH-DS:00, YB-VUR-TK:88}.} ${\bf b}\in \mathbb{R}^{mN}$ is the maximizer of the following information divergence criterion: 
	\begin{align}\label{eq: I-divergence}
	I\!=\!\pi_1\mathbb{E}\left[ \left.\ln \frac{f_{H_1}(y)}{f_{H_2}(y)} \right\vert H_1\right]\!+\!\pi_2\mathbb{E}\left[ \left.\ln \frac{f_{H_2}(y)}{f_{H_1}(y)}\right\vert H_2\right], 
	\end{align}
	where $f_{H_i}(y)$ is the density of $y$ given $H_i$ and $I>0$.
\end{definition}

\smallskip 
\begin{remark}{\bf \textit{(Optimal discriminant vector)}}\label{rmk: optimal discriminant} 
	For any arbitrary vector ${\bf b}$, the $I$-divergence measure \eqref{eq: I-divergence} indicates how well a LD-MAP detector is performing in deciding between $H_1$ and $H_2$. Thus by maximizing \eqref{eq: I-divergence}, we are finding a best detector among the class of LD-MAP detectors parameterized by ${\bf b}$ \cite{LLS:91}. \oprocend
\end{remark}

We now state a lemma that provides us with the algebraic expressions of the MAP detectors associated with the mean shift model, and the LD-MAP detector associated with the covariance shift model.

\begin{proposition}{\bf \textit{(Mean and covariance of ${\bf Y}_{\cal J}$)}}\label{prop: moments of measurements} Let ${\bf Y}_{\cal J}$ and $H_i$ be defined as in \eqref{eq: measurements recorded} and \eqref{eq: input hypothesis}, resp. Then, 
	\begin{align}\label{eq: moments}
	\begin{split}
	{\overline{\boldsymbol \mu}}_i&\!\triangleq \!\mathbb{E}[{\bf Y}_{\cal J}|H_{i}]\!=\!\mathcal{F} \left({\bf 1}_N \otimes {\boldsymbol \mu}_i\right) \text{ and }\\
	{\overline{ \Sigma}}_i&\!\triangleq\!\mathrm{Cov}[{\bf Y}_{\cal J}|H_{i}]\!=\!\mathcal{O}\Sigma_0\mathcal{O}^\transpose+\mathcal{F}\left(I_{N }\otimes{\Sigma}_i\right) \mathcal{F}^{\transpose}+\sigma^2_v I, \\
	\end{split}
	\end{align}
	where, the observability and impulse response matrices are
	\begin{align*}
	\mathcal{O}\!=\!\begin{bmatrix}
	CG \\ C G^2\\ \vdots \\ CG^N
	\end{bmatrix}\text{and }
	\mathcal{F}\!=\!\begin{bmatrix}
	C\Pi & 0 & \ldots & 0\\
	CG\Pi & C\Pi & \ldots & 0\\
	\vdots & \vdots & \ddots & \vdots\\
	CG^{N-1}\Pi & CG^{N-2}\Pi & \ldots & C\Pi
	\end{bmatrix}.
	\end{align*}
\end{proposition}

\smallskip

\begin{lemma}{\bf \textit{(MAP detectors)}}\label{lemma: MAP detector} Let $\pi_1$ and $\pi_2$ be non-zero priors, and define $\gamma=\ln(\pi_1/\pi_2)$. Let ${\bf Y}_{\cal J}$ be  as in \eqref{eq: measurements recorded}, and let $({{\boldsymbol \mu}}_i, {{ \Sigma}}_i)$ and $({\overline{\boldsymbol \mu}}_i, {\overline{ \Sigma}}_i)$ be as in \eqref{eq: input hypothesis} and  \eqref{eq: moments}, resp. 
	\begin{enumerate}
		\item The MAP detector associated with the \textit{mean shift model} ($ {\Sigma}_1={\Sigma}_2$ but $\bs{\mu}_1\ne \bs{\mu}_2$) is given by: 
		\begin{align}\label{eq: MAP mean detector}
		\left(2\, \ol{\bs \mu}_\Delta^\transpose {\overline \Sigma}_c^{-1}\right)  {\bf Y}_{\cal J} \overset{\widehat{H}=H_{2}}{\underset{\widehat{H}=H_{1}}{\gtrless}}2\gamma + \ol{\bs \mu}_\Delta^\transpose{\ol \Sigma}_c^{-1}\left( \ol{\bs \mu}_1+\ol{\bs \mu}_2\right),
		\end{align}
		where $\ol {\bs \mu}_\Delta= \ol {\bs \mu}_2-\ol {\bs \mu}_1$ and $\ol \Sigma_c\triangleq \ol \Sigma_1=\ol \Sigma_1$.
		\item The LD-MAP detector associated with the \textit{covariance shift model} ($ {\Sigma}_1\!\ne\!{\Sigma}_2$ but $\bs{\mu}_1\!=\!\bs{\mu}_2$) is given by: 
		\begin{align}\label{eq: MAP covariance detector}
		\ln \left(\frac{d_1}{d_2}\right)-2\gamma \overset{\widehat{H}=H_{2}}{\underset{\widehat{H}=H_{1}}{\gtrless}}(y-{\bf b}^\transpose{\bs {\ol \mu}}_c)^2\left[\frac{1}{d_2}-\frac{1}{d_1}\right], 
		\end{align}
		where $y={\bf b}^\transpose{\bf Y}_{\cal J}$, $d_i={\bf b}^\transpose \ol \Sigma_i {\bf b}$, and $\ol {\bs \mu}_c\triangleq  \ol {\bs \mu}_1=\ol {\bs \mu}_2$. 
	\end{enumerate}
\end{lemma}

\smallskip 
The detectors \eqref{eq: MAP mean detector} and \eqref{eq: MAP covariance detector} are functions of the sufficient statistics $2\, \ol{\bs \mu}_\Delta^\transpose {\overline \Sigma}_c^{-1}{\bf Y}_{\cal J}$ and $y-{\bf b}^\transpose{\bs {\ol \mu}}_c$, respectively. This
means that, given these statistics, other information in ${\bf Y}_{\cal J}$ is not needed for deciding between $H_1$ and $H_2$. In order to characterize the error probabilities of the detectors in Lemma \ref{lemma: MAP detector}, we make the following assumption:

\begin{assumption}\label{assump: error probability assumption}
	The LTI system \eqref{eq: system} is stable. Further, 
	\begin{enumerate}
		%		\item for any sensor set $\cal J$, measurements are collected after the delay $d_{\cal J}\ne 0$, where $d_{\cal J}\triangleq \inf\left\lbrace d: CG^d\Pi\ne 0\right\rbrace$, 
		\item for the \textit{mean shift model}, $\lim_{N\to \infty}N\|{\bs \mu}_2-{\bs \mu}_1\|_2= c$, where $0<c<\infty$, and $G^k=0$ for some $k\in {\mathbb N}$, and 
		\item for the \textit{covariance shift model}, $\Sigma_1\succ0$ and $\Sigma_2=0$.
	\end{enumerate}
\end{assumption}

\begin{lemma}{\bf \textit{(Error probability: infinite horizon)}}\label{lma: special cases error probability} 
	Let $\pi_1=\pi_2=0.5$ and ${\bf x}[0]=0$. Let $T(z)=C(zI-G)^{-1}\Pi$, where $z \notin \text{spec}(G)$. 
	The error probability of the MAP detector \eqref{eq: MAP mean detector} and the LD-MAP detector \eqref{eq: MAP covariance detector} as $N\to\infty$ are 
	\begin{align} 
	\mbP_{e_m}(\cal J)&\!=\!0.5\,Q_{\cal N}(0.5\,\eta) \text{ and }\label{eq: P_em}\\
	\mbP_{e_v}(\cal J)&\!=\!0.5\left[ 1-Q_{\chi^2}\left(1, \tau \right)\right] +0.5\, Q_{\chi^2}\left(1, \tau R\right)\label{eq: P_ev}, 
	\end{align}
	respectively, where $\tau=\ln R/(R-1)$. The SNRs are
	\begin{align} 
	\eta^2&=N\widetilde {\bs \mu}_{\Delta}^\transpose\left(  [ L^\transpose L+\sigma^2_vI] ^{-1}L^\transpose L\right)\widetilde{ \bs \mu}_{\Delta} \label{eq: eta asymp}\text{ and }\\
	R&=1+{\sigma^{-2}_v}\|T(z)\Sigma_1^{\frac{1}{2}}\|^2_\infty,  \label{eq: Rb asymp}
	\end{align}
	where $L=T(1)\Sigma_c^{\frac{1}{2}}$ and $\widetilde {\bs \mu}_{\Delta}=\Sigma_c^{-\frac{1}{2}}[{\bs \mu}_2-{\bs \mu}_1]$, and $\Sigma_c^{\frac{1}{2}}$ and $\Sigma_1^{\frac{1}{2}}$ are the positive square roots of $\Sigma_c$ and $\Sigma_1$, respectively. 
\end{lemma}

\smallskip 
The assumptions $\pi_i=0.5$ and ${\bf x}[0]=0$ are for the ease of presentation, and the probability expressions can be easily adjusted to include other priors and initial conditions. The assumption $N\|{\bs \mu}_2-{\bs \mu}_1\|_2\to c$ ensures that $\mbP_{e_m}({\cal J})<0.5$. Instead, the assumption $G^k=0$ is to eliminate  the remainder terms in the computation of $\eta$. We emphasize that the only restriction on $k$ is that it should be finite, but can be arbitrarily large. We now state a corollary to the above lemma in which we do not assume $\Sigma_2=0$ in the covariance shift model (see Remark \ref{rmk: other covariance structures}).

\smallskip 
\begin{corollary}{\bf \textit{(SNRs: identical input statistics)}}\label{cor: special cases error probability} 
	Let $H_i$ in \eqref{eq: input hypothesis} be ${\bf w}[k]\sim \mathcal{N}(\mu_i{\bf 1}, \sigma_i^2D)$, where $\mu_i$ and $\sigma_i^2$ are scalars, and $D>0$. For the \textit{covariance shift model} let $\sigma_1^2>\sigma_2^2$. Then,
	\begin{align} 
	\eta^2_s&=\left( N\mu_\Delta^2 \right) {\bs 1}^\transpose [\sigma^2_cL^\transpose L+\sigma^2_vI] ^{-1}L^\transpose L{\bs 1}, \label{eq: identical eta asymp}\\
	R_s&=\frac{\sigma_1^2\|T(z)D^{\frac{1}{2}}\|_\infty+\sigma^2_v}{\sigma_2^2\|T(z)D^{\frac{1}{2}}\|_\infty+\sigma^2_v}\label{eq: identical R asymp}, 
	\end{align}
	where $\mu_c=\mu_i$, $\sigma_c=\sigma_i$, $L=T(1)D^{\frac{1}{2}}$, and $\mu_\Delta=\mu_2-\mu_1$. 
\end{corollary}

The error probabilities for the identical statistics case can be obtained by substituting $\eta_s$ and $R_s$ to $\eta$ and $R$ in \eqref{eq: P_em} and \eqref{eq: P_ev}, respectively. The effect of sensor noise is also evident from the SNR expressions in the above corollary. In particular, by setting $\sigma^2_v=0$ in \eqref{eq: identical eta asymp} and \eqref{eq: identical R asymp}, the probabilities do not depend on the network matrix $G$. 

Notice that the expressions of $\mbP_{e_m}(\cal J)$ and $\mbP_{e_m}(\cal J)$ in above lemma are valid even when $N$ is finite. However, in this case, $\eta$ and $R$ are complicated functions of the adjacency matrix $G$. Instead, the elegance of SNRs in Lemma \ref{lma: special cases error probability} and Corollary \ref{cor: special cases error probability} is that they depend on the adjacency matrix $G$ through the well understood transfer matrix $T(z)$. Thus, when $N \to \infty$, one can easily understand the impact of network structure on the detection performance by analyzing $T(z)$. By interpreting the quadratic function in $\eta$ (or $\eta_s$) and $\|\cdot||_{\infty}$ in $R$ (or $R_s$) as a measure of gain, one expects that higher gains results in minimum error probabilities. This intuition is made precise in the following proposition:

\begin{proposition}\label{prop: error probability vs eta and R}
	$\mbP_{e_m}(\cal J)$ and $\mbP_{e_v}(\cal J)$ are decreasing in the SNRs $\eta$ (or $\eta_s$) and $R$ (or $R_s$), respectively. 
\end{proposition}

The above proposition also helps us to compare the performance of the MAP and LD-MAP detectors associated with different sensor sets. This fact will be exploited greatly in the next section. 

\smallskip 
\begin{remark}{\bf \textit{(LD-MAP detector's error probability for other covariance matrix structures)}}\label{rmk: other covariance structures} We now comment on extending $\mbP_{e_v}(\cal J)$ \eqref{eq: P_ev} for including other covariance matrices. The case $\Sigma_1=0$ and $\Sigma_2>0$ can be handled using the proof of Lemma \ref{lma: special cases error probability}. For the scenario where neither of $\Sigma_1$ or $\Sigma_2$ is zero, if we have $N<\infty$ and $\lambda_{\text{max}}(\ol \Sigma_1\ol \Sigma_2^{-1})>\lambda_\text{min}(\ol \Sigma_1\ol \Sigma_2^{-1})$, then $\mbP_{e_v}({\cal J})$ remains the same as in \eqref{eq: P_ev}, with $R=\lambda_{\text{max}}(\ol \Sigma_1\ol \Sigma_2^{-1})$. For other cases we refer the reader to \cite{LLS:91}. However, the main difficulty in analyzing any of these error probabilities lies in the fact that resulting expressions of SNRs ($R$) are not amenable to analysis. If one assumes $\ol \Sigma_1$ and $\ol\Sigma_2$ to be simultaneously diagonalizable, as is the case with Corollary \ref{cor: special cases error probability}, an expression of $R$ similar to \eqref{eq: identical R asymp} may be obtained.  \QEDB
\end{remark}

\section{Network analysis of the MAP detector}\label{sec: network_analysis}

In this section, we characterize networks for which the MAP detector's performance associated with the sensors that are close to the input nodes is better (or worse) than those of sensors that are farther apart. We distinguish two separate cases when the sensors are without noise ($\sigma^2_v>0$) and with noise ($\sigma^2_v=0$). To make the notion of closeness precise, we introduce the notion of a node cutset. 

\smallskip 
\begin{definition}{\bf \textit{(Node cutset)}}\label{def: sep cut set}
	For the graph $\mathcal G:=(\cal V, \cal E)$ with input nodes $\cal K$, the nodes ${\cal C}_d\subseteq {\cal V}$, with $d>1$, form a \textit{node cutset} if there exist a non empty source set ${\cal S}\subseteq {\cal V}$ and a non empty partitioned set ${\cal P}\subseteq {\cal V}$ such that ${\cal V}={\cal S}\sqcup{\cal C}_d\sqcup{\cal P}$, where $\sqcup$ denotes the disjoint union, and 
	%	
	%	
	%A set of nodes ${\cal C}_{d}\subset \cal V$ in the graph $\mathcal G:=(\cal V, \cal E)$ with $n$ nodes is said to be a \textit{node cutset}, if ${\cal V}={\cal S} \sqcup {\cal C}_d \sqcup \cal P$ such that the following holds true: 
	\begin{itemize}
		\item [(i)] $\cal K\subseteq \cal S$ and $\text{dist}({\cal K}, {\cal C}_d)\geq d$, and 
		\item [(ii)] every path from $\cal S$ to $\cal P$ contains a node in ${\cal C}_d$. 
		%\item [(iii)] removal of nodes ${\cal C}_d$, i.e., ${\cal V}\setminus {\cal C}_d$, renders $\cal G$ disconnected. 
	\end{itemize}
\end{definition}

\smallskip 
The requirement (i) ensures that the node cutset is at least $d$ edges away from the input nodes. To illustrate Definition \ref{def: sep cut set}, consider the network in Fig \ref{fig: cutset}. For the input nodes ${\cal K}=\{1,2\}$, the nodes ${\cal C}_1=\{4, 5, 6\}$ forms a node cutset. However, the nodes $\{5,6,7\}$ ceases to form a node cutset, since they failed to satisfy requirement (ii) in the above definition. 
%\smallskip 
%\begin{assumption}
%	Given a graph $\cal G$, the input node set $\cal K$ and the distance $d$ is predetermined and remains constant. 
%\end{assumption}\smallskip 

\subsection{Noiseless measurements}
In this section, we state our results on network theoretic characterization of the MAP detectors assuming that the measurement noise in \eqref{eq: system} is negligible, i.e., $\sigma^2_v=0$.  It should be noted that, if a result holds true for the general detection problem \eqref{eq: input hypothesis}, we do not state the analogous result for the \textit{mean} and \textit{covariance shift models}.  

\smallskip 
%
%\subsection{Noiseless measurements}
%Although in practice, obtaining perfect measurements is often rare, we still analyze this scenario to get insights into the MAP detector's performance over network systems. 

\begin{theorem}{\bf \textit{(Performance of sensors on the node cutset vs the partitioned set: noiseless measurements)}}\label{thm: network_theoretic_result_noiseless}  Consider the general detection problem \eqref{eq: input hypothesis}. Let ${\cal C}_d$ and $\cal P$ be as in Definition \ref{def: sep cut set}, and assume that the measurements from both these node sets are noiseless ($\sigma^2_v=0$). Associated with these measurements, let $\mathbb{P}_e\left( {\mathcal{C}_d}\right)$ and $\mathbb{P}_e\left( {\mathcal{P}}\right)$ be the respective error probabilities that are computed using \eqref{eq: general error probability}. Then, $\mathbb{P}_e\left( {\mathcal{C}_d}\right)\leq \mathbb{P}_e\left( {\mathcal{P}}\right)$. 
\end{theorem}

This comparison result is a mere consequence of the following well known result in the binary hypotheses detection problem, known as \textit{theorem of irrelevance} \cite{JM-IM:65} and the \textit{invariance of MAP decision rule} \cite{JM-DC:78}. 

\begin{lemma}{\bf \textit{(Error probability of the MAP detector: dependent measurements)}}\label{lma: dependent_measurements_poe} 
	Let $M_1$ and $M_2$ be any two arbitrary simple hypotheses with non-zero priors. Let $\delta_1$ be the error probability of a MAP detector relying on the measurement ${\bf Y} \in \mathbb{R}^{p_1}$, and $\delta_2$ be such a quantity associated with the measurement ${\bf Z}=g({\bf Y})+{\bf v}$, where $g(.) : \mathbb{R}^{p_1} \to \mathbb{R}^{p_2}$ and ${\bf v}$ is stochastically  independent of the hypotheses. Then, $\delta_1\leq \delta_2$.
\end{lemma}

From Lemma \ref{lma: dependent_measurements_poe}, it also follows that Theorem \ref{thm: network_theoretic_result_noiseless} holds true even (i) for the case of  non-Gaussian input and measurements (provided that the joint density exists), and (ii) if the set $\cal P$ is replaced with ${\cal P} \cup \widetilde {\cal C}_d$, where $ \widetilde {\cal C}_d\subseteq {\cal C}_d$.

Theorem \ref{thm: network_theoretic_result_noiseless} implies that, in the absence of noise, nodes near the input location
achieve better detection performance compared
to those far away from the inputs, \textit{irrespective of the edge weights in the adjacency matrix $G$ and the measurement horizon $N$}. Here, the notion of closeness is to be understood in the sense of \textit{node cutsets}, since, $d\leq \text{dist}({\cal K, \cal C}_d)<\text{dist}({\cal K, \cal P})$. Thus, if node cutsets exist in a graph and the measurements are noiseless, one should always place sensors on the cutsets. Thus, if a budget is associated with the sensor placement, it makes sense to find a cutset ${\cal C}_d$ of minimum cardinality. 

\begin{proposition}{\bf \textit{(Error probability of the oracle detector)}}\label{prop: oracle}  Consider the general detection problem \eqref{eq: input hypothesis}, and let $\delta_1$ be the error probability of a MAP detector which can directly access the inputs ${\bf w}[k]$, $k=0,\ldots,N$. For any sensor set $\cal J$, let $\delta_2$ and $\delta_3$ be the error probabilities associated with the noiseless $(\sigma^2_v=0)$ and noisy $(\sigma^2_v>0)$ measurements ${\bf Y}_{\cal J}$ \eqref{eq: measurements recorded}, respectively. Then, $\delta_1\leq \delta_2\leq \delta_3$. 
\end{proposition}

\smallskip 
Proposition \ref{prop: oracle} states that sensor noise degrades the performance of the MAP detector (this fact is also illustrated in Example \ref{exam: classical detection setup}). It also implies that measuring the inputs directly is always better than measuring the noisy/noiseless states (dynamics) of the nodes. Of course, given this fact, it is always beneficial to place the sensors at the input nodes, rather than dealing with the \textit{node cutsets} and the {\it partitioned sets}. 

\subsection{Noisy measurements}

We now consider the case of noisy measurements ($\sigma^2_v>0$). Notice that our results will be specific to the MAP and LD-MAP detectors associated with the \textit{mean} and \textit{covariance shift models}, respectively. Possible extensions to the general detection problem \eqref{eq: input hypothesis} are mentioned in the remarks. We now introduce some additional notation. For a cutset ${\cal C}_d$, let ${\bf x}_c[k]$, ${\bf x}_s[k]$, and ${\bf x}_p[k]$ denote the states of the node sets ${\cal C}_d$, $\cal S$, and $\cal P$, respectively. Let $M$ be a permutation matrix such that ${\bf x}[k]=M[{\bf x}_s[k]^\transpose, {\bf x}_c[k]^\transpose, {\bf x}_p[k]^\transpose]^\transpose$, where ${\bf x}[k]$ is the state vector of \eqref{eq: system}. Then, from \eqref{eq: system} it also follows that 
\begin{align}\label{eq: decomposed system}
\begin{bmatrix}
{\bf x}_s[k+1]\\
{\bf x}_c[k+1]\\
{\bf x}_p[k+1]
\end{bmatrix}\!=\!
\underbrace{\begingroup
	\setlength\arraycolsep{1.9 pt}\begin{bmatrix}
	G_{ss} & G_{sc} & 0\\
	G_{cs} & G_{cc} & G_{cp}\\
	0 & G_{pc} & G_{pp}
	\end{bmatrix}\endgroup}_{M^{-1}GM}\begin{bmatrix}
{\bf x}_s[k]\\
{\bf x}_c[k]\\
{\bf x}_p[k]
\end{bmatrix}\!+\!\underbrace{\begin{bmatrix}
	{\bf w}_s[k]\\{\bf 0} \\{\bf 0}
	\end{bmatrix}}_{M^{-1}\Pi{\bf w}[k]}. 
\end{align}
From the above relation, note that the states of ${\cal C}_d$ serve as an input for the states of partitioned nodes set $\cal P$, i.e., 
\begin{align}\label{eq: subsystem}
{\bf x}_p[k+1]=G_{pp}{\bf x}_p[k]+G_{pc}{\bf x}_c[k]. 
\end{align}

Based on the transfer function matrix of subsystem \eqref{eq: subsystem}, we now state a result that is analogous to Theorem \ref{thm: network_theoretic_result_noiseless}, for the case $\sigma_v^2>0$. 
%Finally, let $\widetilde G=[G_{pp}\,\, G_{pc}]$. 
\begin{theorem}{\bf \textit{(Performance of sensors on the node cutset vs the partitioned set: noisy measurements)}}\label{thm: network_theoretic_result_noisy} Let $G_{pp}$ and $G_{pc}$ be as in \eqref{eq: decomposed system}, and assume that $\mathrm{spec}(G_{pp})\cap \{z\in\mathbb{C}: |z|=1\}=\phi$. Let $\overline\rho(z)$ and $\underline\rho(z)$ be the maximum and minimum singular values of $T_s(z)=(zI-G_{pp})^{-1}G_{pc}$, respectively.
	Let $\mbP_{e_m}({\cal C}_d)$ in \eqref{eq: P_em}  and $\mbP_{e_v}({\cal C}_d)$ in \eqref{eq: P_ev} be the error probabilities obtained using the noisy measurements ($\sigma_v^2>0$) from the cutset ${\cal C}_d$. Instead, let $\mbP_{e_m}({\cal P})$ and $\mbP_{e_v}({\cal P})$ be the error probabilities associated with the partitioned set $\cal P$. Then we have: 
	\begin{enumerate}
		\item  [1a)] If $\overline\rho(1)\leq 1$, then $\mbP_{e_m}({\cal C}_d)\!\leq\!\mbP_{e_m}({\cal P})$. 
		\vskip 0.2em
		\item  [1b)] If $\underline\rho(1)>1$, then $\mbP_{e_m}({\cal C}_d)\!>\!\mbP_{e_m}({\cal P})$.
		\vskip 0.2em
		\item [2a)] If $\sup_{|z|=1}\overline\rho(z)\leq1$ then $\mbP_{e_v}({\cal C}_d)\!\leq\!\mbP_{e_v}({\cal P})$. 
		\vskip 0.2em
		\item [2b)] If $\inf_{|z|=1}\underline\rho(z)>1$, then $\mbP_{e_v}({\cal C}_d)\!>\!\mbP_{e_v}({\cal P})$.
	\end{enumerate}
\end{theorem} 

\smallskip 
Hence, in the presence of noise, depending upon the entries in the matrix $[G_{pp}\, G_{pc}]$, measuring the cutset ${\cal C}_d$ might not be always optimal for the purposes of the detection. Instead, in the noiseless case, Theorem \ref{thm: network_theoretic_result_noiseless} states that measuring the cutset is always optimal, irrespective of the entries in $G$. We now explain the reason behind this contrasting behaviors.

Notice that, the quantities $\sup$ and $\inf$ of $\overline\rho(z)$ and $\underline\rho(z)$ in Theorem \ref{thm: network_theoretic_result_noisy}, respectively, are the maximum and minimum  input to output gains of the transfer function matrix $T_s(z)$, associated with the system \eqref{eq: subsystem}. Theorem \ref{thm: network_theoretic_result_noisy} says that, if the gain between the states ${\bf x}_c[k]$ and the states ${\bf x}_p[k]$ is high (low), the detection performance with sensors in $\cal P$ should be better (worse) that that of ${\cal C}_d$. In fact, recall from Lemma \ref{lma: special cases error probability} that the detectors associated with the noisy measurements of ${\cal C}_d$ and ${\cal P}$, respectively, depends on the SNRs of ${\bf x}_c[k]$ and ${\bf x}_c[k]$ (plus the sensor noise), respectively. Since ${\bf x}_p[z]=T_s(z){\bf x}_c[z]$, it is clear that the SNRs are influenced by the gains of $T_s(z)$. In particular, a higher gain increases the SNR of the detector associated with ${\cal P}$, which results in a better performance compared to the detector associated with that of ${\cal C}_d$. 

The above reasoning also holds in the case of noiseless measurements, however, the transfer function gain do not influence MAP detector's performance. In fact, this gain gets canceled in the error probability computations (this can be clearly seen in Example \ref{exam: classical detection setup} by interpreting $\alpha$ as the gain). Theorem \ref{thm: network_theoretic_result_noisy} provides conditions for placing sensors on or away from the cutset nodes. For general adjacency matrix, one needs to rely on the software (based on LMI based inequalities) to validate those conditions. However, for non-negative adjacency matrices, the conditions for placing (or not) sensors on the cutset nodes can be stated based on algebraic conditions on the entries of the adjacency matrix. In fact, we have the following result: 

\smallskip 
\begin{lemma}{\bf \textit{(Non-negative adjacency matrix)}}\label{lma: non-negative matrix}
	Let the matrix $G$ in \eqref{eq: system} be non-negative, and $\widetilde G=[G_{pp}\, G_{pc}]\in \mathbb{R}^{m_1\times n_1}$, where $G_{pp}$ and $G_{pc}$ are defined in \eqref{eq: subsystem}. 
	\begin{enumerate}
		\item  If $\|\widetilde G\|_{\infty}\!\leq\! 1/\sqrt{m_1}$, then we have $\mbP_{e_m}({\cal C}_d)\!\leq\!\mbP_{e_m}({\cal P})$ and $\mbP_{e_v}({\cal C}_d)\!\leq\!\mbP_{e_v}({\cal P})$. 
		\item  If $n_1=1$, and all row sums of $\widetilde G$ are greater than one, then $\mbP_{e_m}({\cal C}_d)\!\geq\!\mbP_{e_m}(\widetilde {\cal P})$ and $\mbP_{e_v}({\cal C}_d)\!\geq\!\mbP_{e_v}(\widetilde {\cal P})$, where $\widetilde {\cal P}\subseteq {\cal P}$. 
	\end{enumerate}
\end{lemma}

The inequality $\mbP_{e_m}({\cal C}_d)\!\leq\!\mbP_{e_m}({\cal P})$ can be obtained even without the non-negativity assumption on $G$. However, this might not be true for the case of $\mbP_{e_v}(\cdot)$. Thus, by ensuring that the maximum row sum of $\widetilde G$ is bounded by $1/\sqrt{m_1}$ (here $m_1$ refers to the cardinality of the partitioned set $\cal P$), one can guarantee that the detection performance of sensors on the cutset is always superior than that of the sensors on the partitioned nodes. The assumption $n_1=1$ in part 2) of above lemma implies that $\text{card}({\cal C}_d)=1$. For arbitrary $n_1$, the condition row sums of $\widetilde G$ greater than one may not be sufficient, and more assumptions on $G$ are required to handle this case. For instance, when $G$ is a diagonally dominant matrix, required sufficient conditions can be obtained using the lower bounds in \cite{JMV:75}. Finally, we notice that the bounds presented in Lemma \ref{lma: non-negative matrix} depends on the cardinality of the node sets, and hence, our results on networks with non-negative edge weights may be conservative when these cardinalities are large. 

{The network-theoretic analysis of the MAP and the LD-MAP
  detectors developed in this section can also be used to inform the
  placement of sensors for detection. The results show that sensors
  placed close to the stochastic inputs are effective for batch
  detection. More precisely, measurements on separating cutsets of the
  network necessarily outperform downstream sensing strategies. Thus,
  a strategy for the placement of few sensors is to find small node
  cutsets that isolate the input nodes. The design of these algorithms
  falls outside the scope of this paper and is left as the subject of
  future research.}

\smallskip 
\begin{remark}{\bf \textit{(Extension of network theoretic results to the other detectors: noisy measurements)}}\label{rmk: non-Gaussian} In the cases where the analytical error probability calculation is difficult, eg., the general Gaussian or non-Gaussian detection problem and the covariance shift model with arbitrary covariance matrix structures, one relies on the Chernoff type bounds (for eg., see \cite{HLV:04}) to quantify the detection performance. In both the cases, i.e., evaluating the performance directly or via bounds, Theorem  \ref{thm: network_theoretic_result_noisy} holds true for any detector whose performance (resp. bounds) is monotonically increasing in $\|T(z)=C(zI-G)^{-1}\Pi\|_M$, for some suitable $M\succ 0$. For instance, the Chernoff bounds on the error probability of the general Gaussian detection problem \eqref{eq: input hypothesis} depend on the moment generating function (mgf) of the sufficient statistic of the MAP detector, which ultimately depends on the filtered mean and covariance matrices \eqref{eq: moments}, and our analysis becomes applicable. In the non-Gaussian case, the mgf might depend on other moments as well, and extending our analysis to this case will be challenging.  \QEDB
\end{remark}

\subsection{Single input single output (SISO) line networks}\label{subsec: toeplitz line networks}  

In this section, we validate our cutset based results, that we presented in previous section, for the case of line networks by explicitly expressing the error probabilities as a function of the entires of $G$, and then compare the performance of sensors on ${\cal C}_d$ versus sensors on ${\cal P}$. We restrict our attention to the SISO systems. 

We assume that a stochastic input enters the network through a fixed node $q\in \{1,\ldots,n\}$, and we allow any node $l\in \{1,\ldots,n\}$ with $\text{dist}(l,q)\geq d$ for sensor placement. For this setup, we assume that probabilities $\mbP_{e_m}(l)$ and $\mbP_{e_v}(l)$ are obtained by substituting the SNRs $\eta_s$ \eqref{eq: identical eta asymp} and $R_s$ \eqref{eq: identical R asymp} in the expressions of \eqref{eq: P_em} and \eqref{eq: P_ev}, respectively. Notice that, in contrast to the previous analysis, in which we assume $\Sigma_2=0$ (see Assumption \ref{assump: error probability assumption}), in this section we do not assume $\sigma_2^2=0$ in $R_s$. For the ease of presentation, we assume the cutset to be a singleton set, i.e., ${\cal C}_d=\{j\}$. The following proposition is an extension of Lemma \ref{lma: non-negative matrix} for our SISO system setup with the revised error probabilities. 

\begin{figure}
	\centering
	\includegraphics[width=1\linewidth]{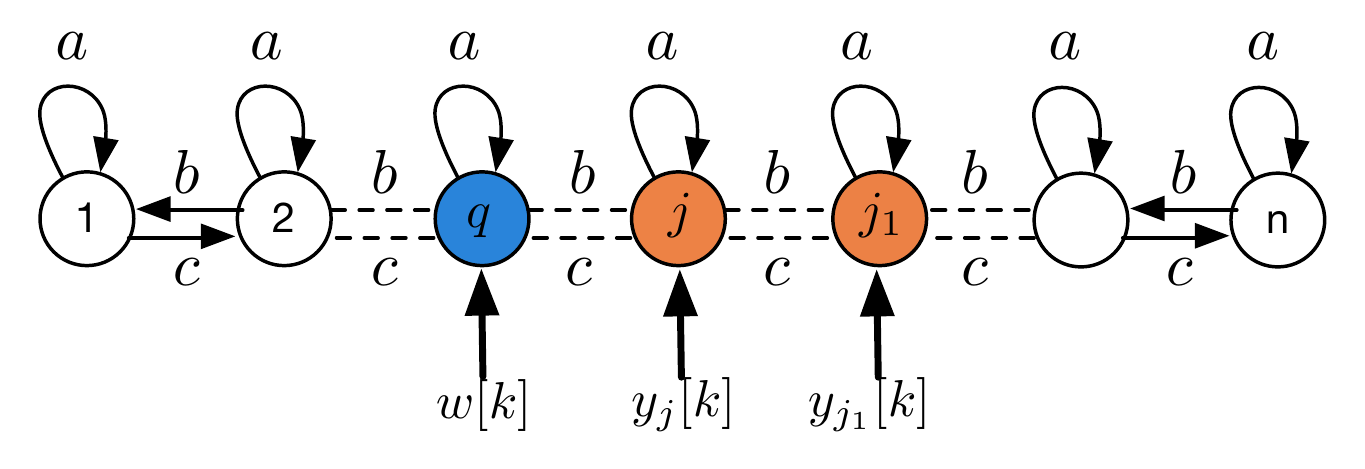}
	\caption{{\footnotesize Toeplitz line network with $n$ nodes. The $q$-th node is injected with the input, and the $j$-th node represents the cutset node.}}
	\label{fig: toeplitz}
\end{figure}

\begin{proposition}\label{ref: prop siso} Let $\widetilde G$ be as in Lemma \ref{lma: non-negative matrix}, and $\sigma^2_v>0$. Let $\{j\}$ and $\cal P$ be the cutset and partitioned sets, resp. If $\|\widetilde G\|_{\infty}\leq 1$, then for any $j_1 \in \cal P$, we have $\mbP_{e_m}(j)\!\leq\!\mbP_{e_m}({j_1})$ and $\mbP_{e_v}(j)\!\leq\!\mbP_{e_v}(j_1)$. The opposite inequality holds true if all row sums of $\widetilde G$ are greater than one. 
\end{proposition}

 The proof of above proposition is similar to the proof of Lemma \ref{lma: non-negative matrix} and hence, the details are omitted.  By not resorting to any proof techniques, i.e, the functional dependence arguments, that we used in previous section, we now validate assertions in above proposition by expressing the error probability in terms of the entries in the matrix $G$. To this aim, we consider a line network (see Fig. \ref{fig: toeplitz}), whose adjacency matrix is given by the following matrix: 
\begin{align}\label{eq: Toeplitz adjacency matrix}
G&=\begin{bmatrix}
a & b & 0 & \cdots & 0 & 0\\
c & a & b & \cdots & 0 & 0\\
\vdots & \vdots & \vdots & \ddots & \vdots & \vdots \\
0 & 0 & 0 & \cdots & a & b\\
0 & 0 & 0 & \cdots & c & a
\end{bmatrix}_{n \times n},
\end{align}
where, $a,b,c \in \mathbb{R}_{\geq 0}$. We let the cutset node $j$ be located on the right of the input node $q$, i.e., $1\leq q<j<n$ (see Fig \ref{fig: toeplitz}). The case when $j$ is to the left side of the input node $q$ follows similar analysis. Thus, we have the partitioned set ${\cal P}=\{j+1,\ldots,n\}$. We now show that, for any $l \in \cal P$, the error probabilities $\mbP_{e_m}(l)$ and $\mbP_{e_v}(l)$ are greater or smaller than those of the cutset node $j$. The following proposition helps us achieve the required goal: 

\begin{proposition}\label{prop: top monontone}
	Let $G$ be as in \eqref{eq: Toeplitz adjacency matrix} and $\overline{\lambda}(G)<1$. Let $|(I-G)^{-1}_{l,q}|$ be the absolute value of $(l,q)$-th entry of $(I-G)^{-1}$. Let $\widetilde G$ be as in Lemma \ref{lma: non-negative matrix}. Then, we have:
	\begin{itemize}
		\item [i)] If $\|\widetilde G\|_{\infty}<1$, then $|(I-G)^{-1}_{q,q}|\geq |(I-G)^{-1}_{q+1,q}| \cdots \geq |(I-G)^{-1}_{n,q}|$.
		\item [ii)] If all row sums of $\widetilde G$ are greater than one, then $|(I-G)^{-1}_{lq}|\geq|(I-G)^{-1}_{qq}|$ for all $q<l\leq n$. If $b=0$, we have $|(I-G)^{-1}_{q+1,q}|\geq |(I-G)^{-1}_{q+2,q}| \cdots \geq |(I-G)^{-1}_{n,q}|$.
	\end{itemize}
\end{proposition}

\smallskip 
For a fixed input $q$, above proposition characterizes the qualitative behavior of the input-to-output transfer function gains associated with different output nodes. This fact can be easily seen by expressing $|(I-G)^{-1}_{lq}|$ as $|{\bf e}_l^\trans(I-G)^{-1}{\bf e}_q|$. For the case of Toeplitz line networks, the assertion in Proposition \ref{ref: prop siso} is now an easy consequence of Proposition \ref{prop: error probability vs eta and R} and \ref{prop: top monontone}. In particular, if $b=0$ and $a+c>1$, Proposition \ref{prop: top monontone} also implies that, the node that is farthest from the input has better detection performance than any other node, including the cutset node. Similarly, assertion in Theorem \ref{thm: network_theoretic_result_noiseless} can be verified by letting $\sigma^2_v=0$. 
%then the sequences are $\{\mbP_{e_m}(l)\}_{l=q+1}^n$ and $\{\mbP_{e_m}(l)\}_{l=q+1}^n$ are decreasing (increasing). This result implies that, for the SISO Toeplitz line networks, if the row sums of $G$ are smaller (or greater) than unity, 
 
The procedure illustrated above, evaluating the error probabilities via the entries of $(I-G)^{-1}$, becomes tedious and might not be even possible for arbitrary network structures. In such situations, one can use the proof techniques presented in Section \ref{sec: network_analysis} for understanding the detection performance of sensors on networks.  

\section{Simulation results}\label{sec: numerical studies} 
In this section, we present numerical simulations to validate the effectiveness of our cutset based characterization of MAP detection performance on networks, for the case of noisy measurements. 
\begin{figure}
	\centering
	\includegraphics[width=1.0\linewidth]{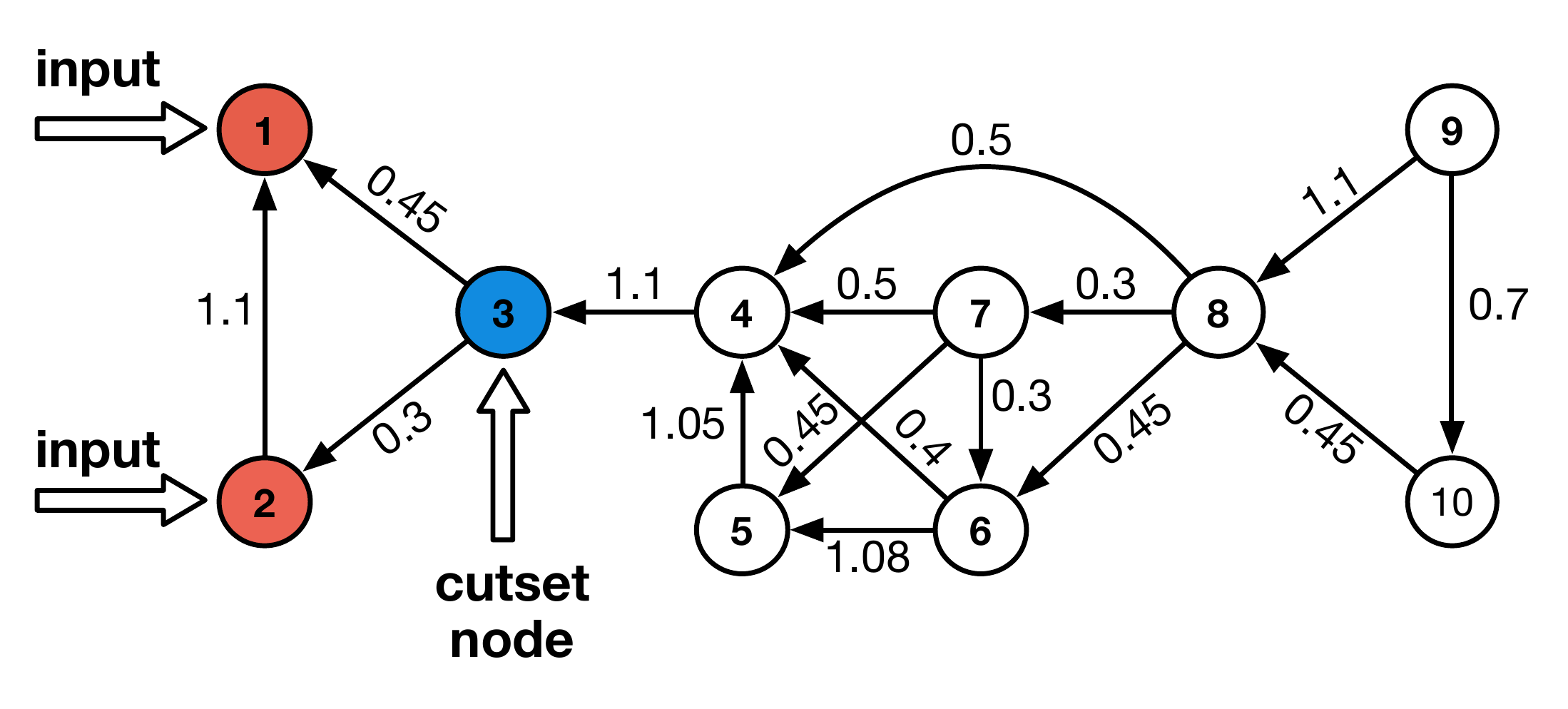}
	\caption{\footnotesize The graph of a network consisting of $10$ nodes. The nodes that are to the right of the cutset node $\{3\}$ form the partitioned set. Instead, nodes $1$ and $2$ form the source set.}
	\label{fig: synthetic}
\end{figure}

\textit{(Detection performance of sensors on the partitioned nodes is better than that of the sensors on the cutset nodes)}: For this scenario, consider the network in Fig \ref{fig: synthetic}. The network has 10 nodes, with $1$ and $2$ being the input nodes, $\mathcal{C}_d=\{3\}$ is the cutset node, and $\mathcal{P}=\{4,\ldots,10\}$ is the partitioned node set. The adjacency matrix of this network is nilpotent, and as a result, system \eqref{eq: system} evolving on this network will have a short memory (in fact $G^{10}=0$). By short (resp. long) memory, we mean that the current state of the network depends on few (resp. several) past states.  For the mean shift model, the input ${\bf w}_i[k]\sim \mathcal{N}(\mu_i{\bf 1},\sigma_i^2I_{2\times 2})$, where $\mu_1=2$, $\mu_2=1$, and $\sigma^2_2=\sigma^2_2=1.5$. Instead, for the  covariance shift model, the input\footnote{the choice of zero mean is arbitrary, since, the LD-MAP detector's error probability do not depend on the mean; see Lemma \ref{lma: special cases error probability}.} ${\bf w}_i[k]\sim \mathcal{N}({\bf 0},\sigma_i^2I_{2\times 2})$, where $\sigma_1^2=2.0$ and $\sigma_2^2=1.0$. In both the models, $N=200$ and the sensor noise variance $\sigma^2_v=1.2$. 

\begin{figure}
	\begin{subfigure}[b]{0.5\linewidth}
		\centering
		\includegraphics[width=1.0\textwidth]{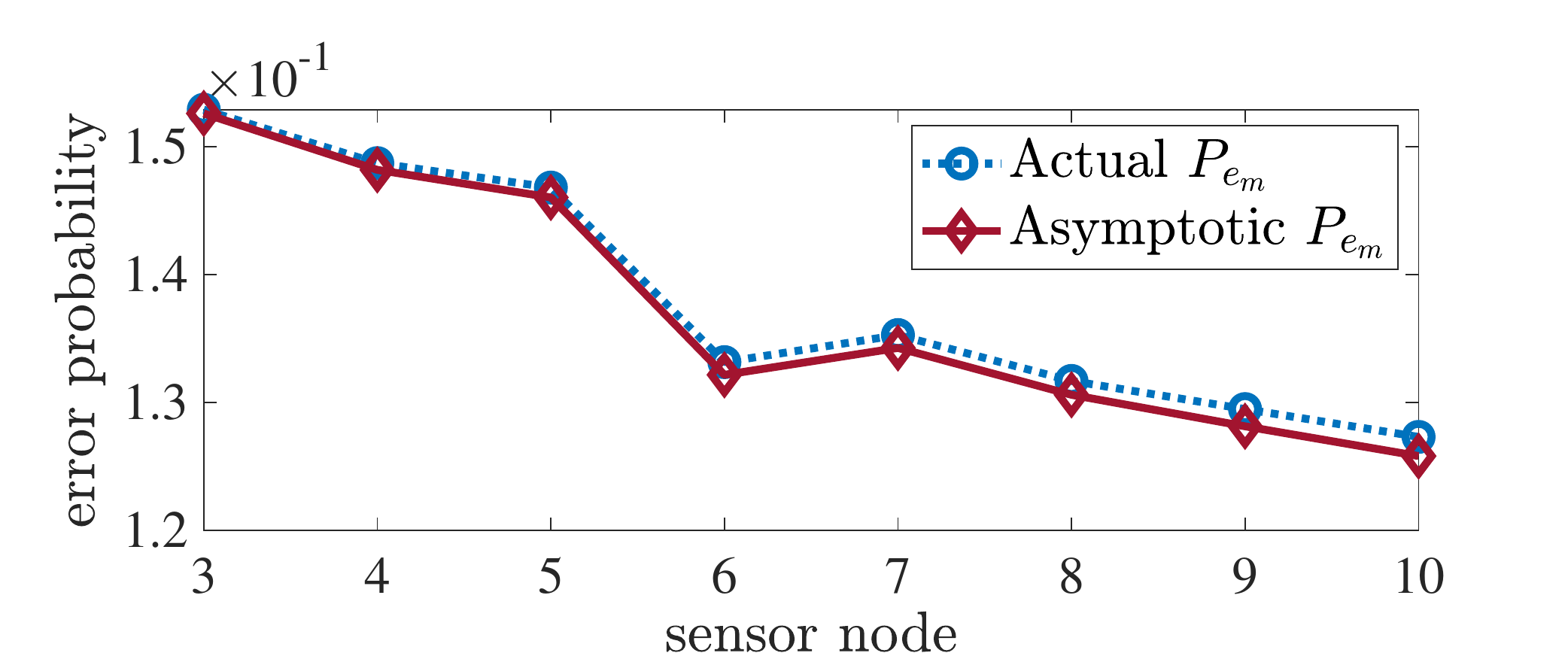} 
		\caption{mean model: short memory} 
		\label{fig: synthetic_a} 
	\end{subfigure}%% 
	\begin{subfigure}[b]{0.5\linewidth}
		\centering
		\includegraphics[width=1.0\textwidth]{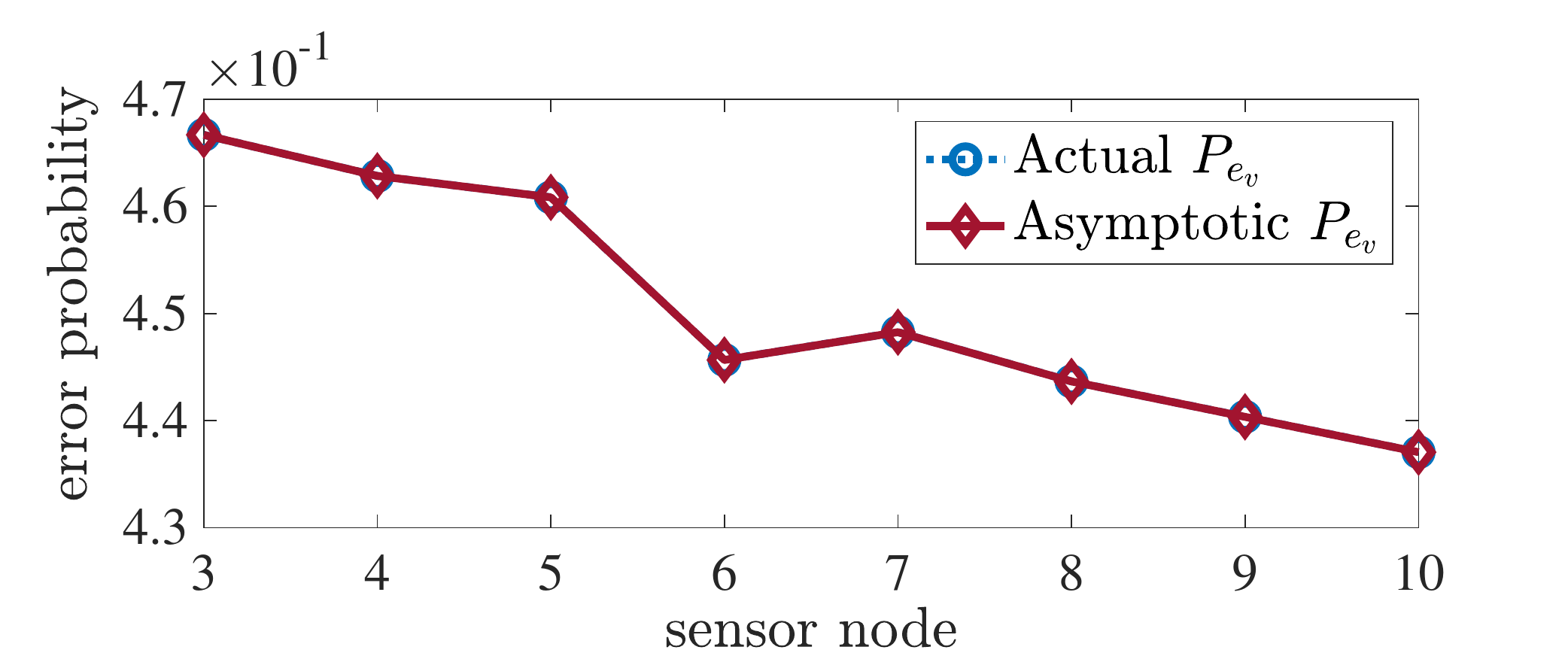} 
		\caption{covariance model: short memory} 
		\label{fig: synthetic_b} 
	\end{subfigure} 
	\begin{subfigure}[b]{0.5\linewidth}
		\centering
		\includegraphics[width=1.0\textwidth]{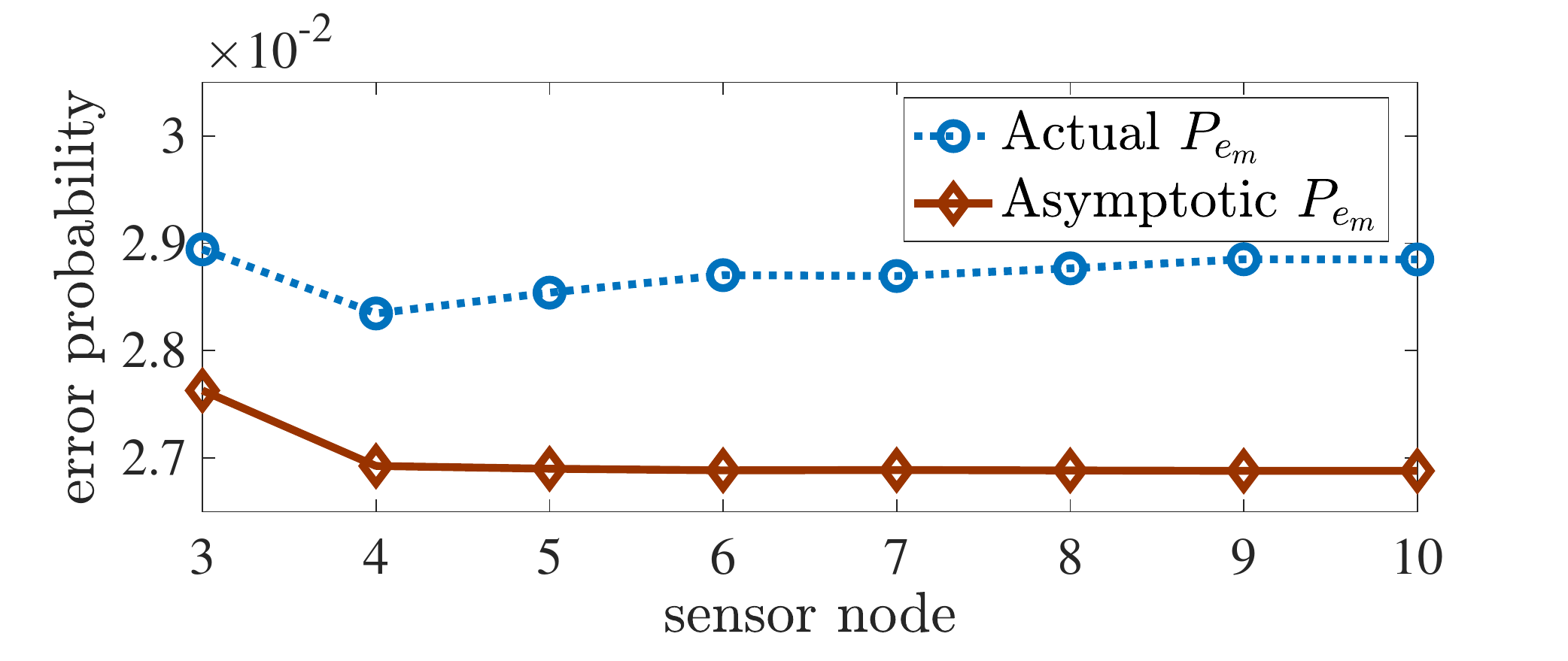} 
		\caption{mean model: long memory} 
		\label{fig: synthetic_c} 
	\end{subfigure}%%
	\begin{subfigure}[b]{0.5\linewidth}
		\centering
		\includegraphics[width=1.0\textwidth]{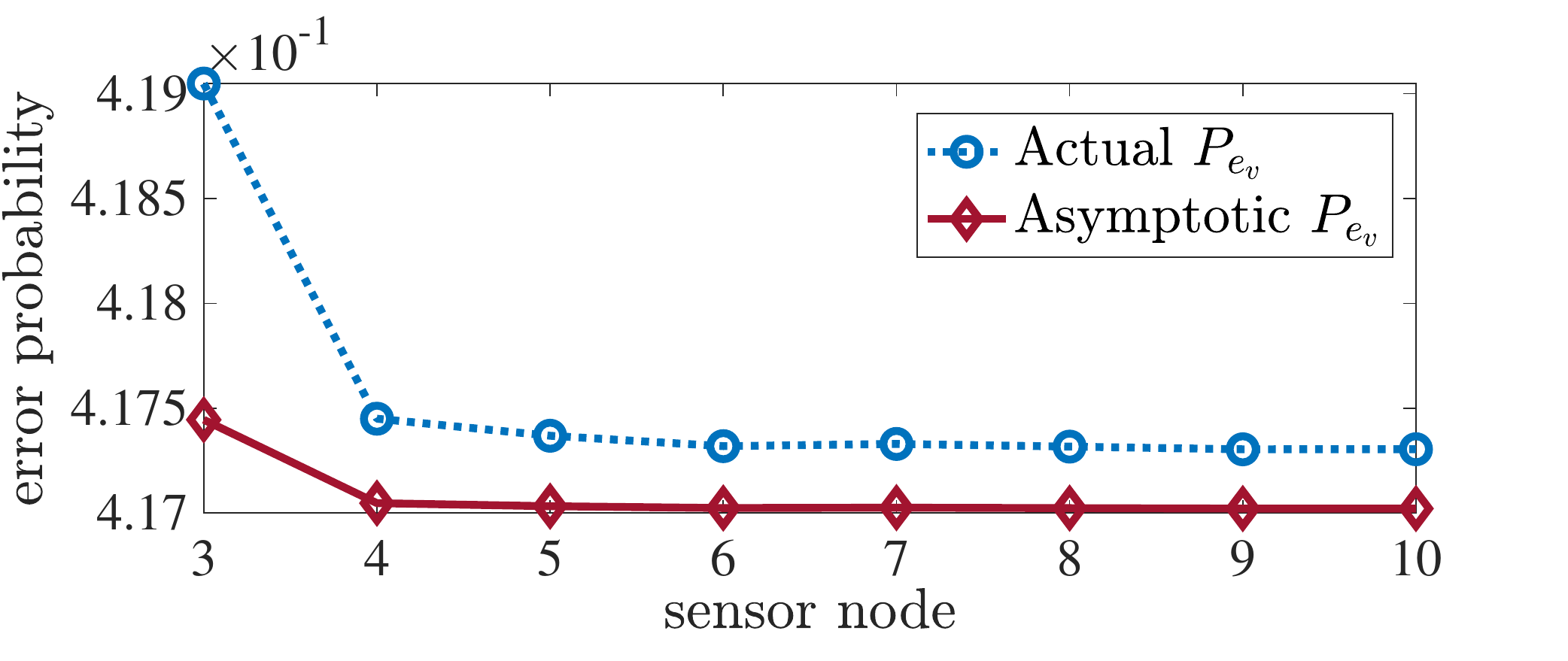} 
		\caption{covariance model: long memory} 
		\label{fig: synthetic_d} 
	\end{subfigure} 
	\caption{\footnotesize Actual and asymptotic error probabilities (Lemma \ref{lma: special cases error probability}) of the MAP and LD-MAP detectors associated with various nodes of the network shown in Fig. \ref{fig: synthetic}. The panels (a) and (b) corresponds to the adjacency matrix that results in the shorter memory of the network dynamics \eqref{eq: system}. Instead, panels (c) and (d) are associated with an adjacency matrix that results in the longer memory of the network dynamics. The error probability associated with each node in the partitioned set $\mathcal{P}=\{4,\ldots,10\}$ is less than that of the cutset node ${\cal C}_d=\{3\}$. This result is consistent with Lemma \ref{lma: non-negative matrix}, because all row sums of submatrix $\widetilde G$ are greater than one.}
	\label{fig7} 
\end{figure}

Fig. \ref{fig: synthetic_a} and Fig. \ref{fig: synthetic_b} illustrates the actual and asymptotic error probabilities of the mean and covariance shift models, respectively. The error probabilities are computed using the formulas in Lemma \ref{lma: special cases error probability}. In particular, for the asymptotic case, we use the SNRs in Corollary \ref{cor: special cases error probability}. In both figures, the error probability associated with the cutset node is greater than that of any node in the partitioned set. This must be the case since $G\geq0$, and the row sums of the submatrix $\widetilde G$ are greater than one (see Lemma \ref{lma: non-negative matrix}). 

The error between the asymptotic and actual error probabilities in Fig. \ref{fig: synthetic_a} and Fig. \ref{fig: synthetic_b}  is almost negligible, even when $N$ is not large. This is because the adjacency matrix $G$ is a nilpotent matrix, and as a result, the difference between the actual and asymptotic SNRs is minimum. However, this might not be the case when $G$ has long memory, i.e., $G^k\approx0$ only for a very large $k$. For $N=800$, Fig. \ref{fig: synthetic_c} and Fig. \ref{fig: synthetic_d} illustrate this scenario for the  
network that is obtained by modifying some edges of the network in Fig. \ref{fig: synthetic}, such that $G^k\approx0$ for very large $k$. 

\textit{(Detection performance of sensors on the cutset nodes is better than that of the sensors on the partitioned nodes)}: 

\begin{figure}
	\centering
	\includegraphics[width=0.9\linewidth]{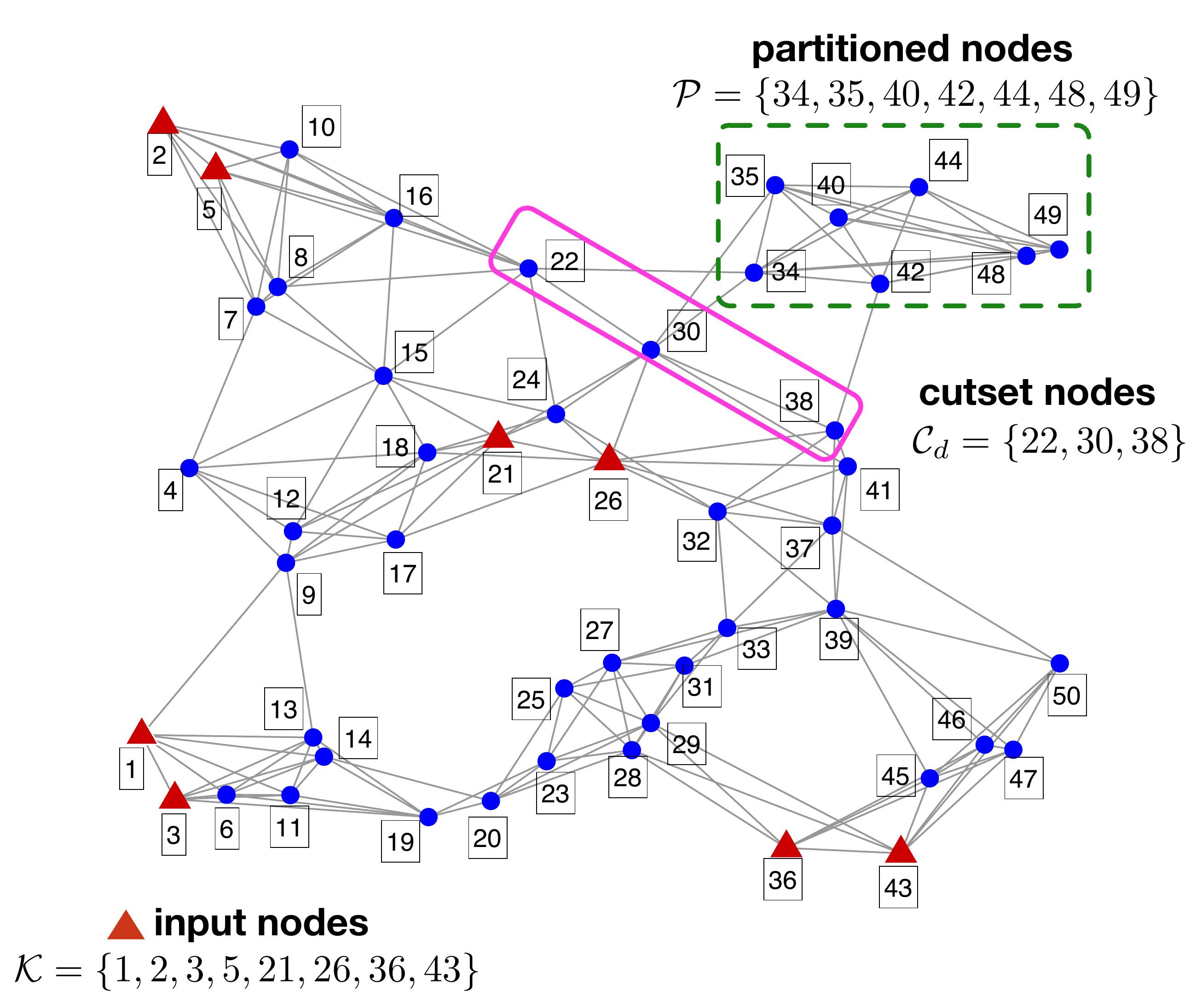}
	\caption{\footnotesize Graph associated with a randomly generated network consisting of $50$ nodes \cite{PN-JP-DS-LM-VK-PV-DH:14}. A total of $8$ nodes are subjected to stochastic inputs. Instead, sensors are placed on the cutset nodes and the partitioned nodes that are not collocated with the input nodes.}
	\label{fig: sensor_network}
\end{figure}

 Consider the network shown in Fig \ref{fig: sensor_network}. The network has 50 nodes among which $\mathcal{K}=\{1,2,3,5,21,26,36,43\}$ are the input nodes. The cutset $\mathcal{C}_d=\{22,30,38\}$ separates $\cal K$ from the partitioned set $\mathcal{P}=\{34,35,40,42,44,48,49\}$. For the mean shift model, the input ${\bf w}_i[k]
\sim\mathcal{N}(\mu_i{\bf 1},\sigma_i^2I_{8})$, where $\mu_1=2$, $\mu_2=1$, and $\sigma^2_2=\sigma^2_2=1.5$, and $\sigma^2_v=1.2$. Instead, for the covariance shift model, the input ${\bf w}_i[k]\sim \mathcal{N}({\bf 0},\sigma_i^2I_{8})$, where $\sigma_1^2=25.0$, $\sigma_2^2=0.1$, and $\sigma^2_v=0.5$. In both the models, $N=200$. 

Consider all possible subsets of $\mathcal{C}_d\sqcup \mathcal{P}$ whose cardinalities are same as that of the cutset $\mathcal{C}_d$. It is easy to see that there are $120$ such sets. For each of these sets, we associate a label $\mathcal{J}_\text{ind}$, where $\text{ind}\in \{1,\ldots,120\}$. The labels are given based on a decreasing order of the error probabilities associated with the subsets. In Fig. \ref{fig: sensor_mean} and Fig. \ref{fig: sensor_cov}, we show the actual and asymptotic error probabilities of the mean and covariance shift models, respectively. In both figures, the error probability associated with the ${\mathcal {C}}_d$ is lesser than that of any $\mathcal{J}_\text{ind}$. This must be the case because $G\geq0$, and the row sums of the submatrix $\|\widetilde G\|_{\infty}<1/\sqrt{7}=0.3780$ (see Lemma \ref{lma: non-negative matrix}). 

\begin{figure}[h!]
	\begin{subfigure}[b]{0.95\linewidth}
		\centering
		\includegraphics[width=1.0\textwidth]{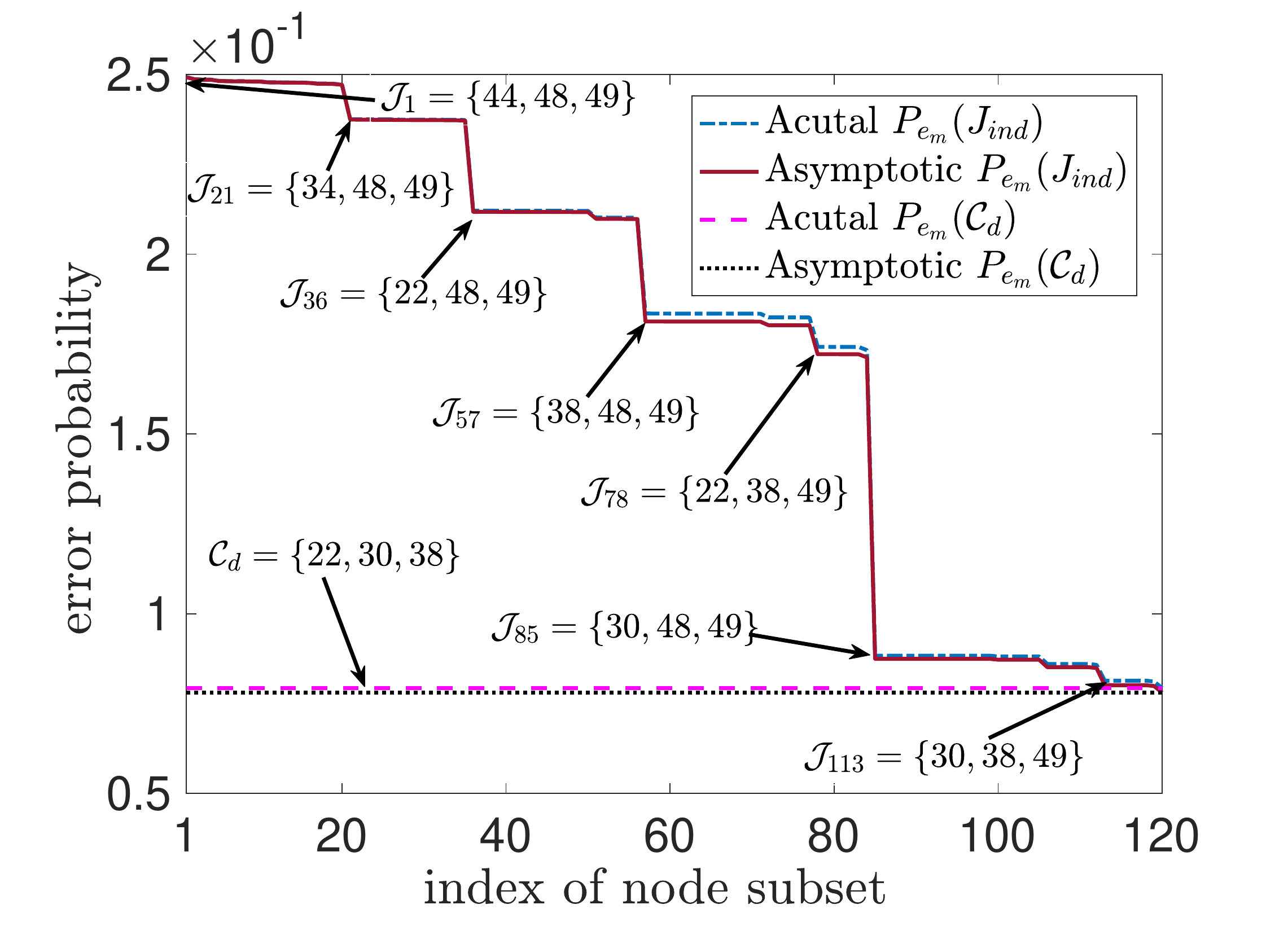}
		\caption{\footnotesize mean shift model.}
		\label{fig: sensor_mean}
	\end{subfigure}
	\begin{subfigure}[b]{0.95\linewidth}
		\centering
		\includegraphics[width=1.0\textwidth]{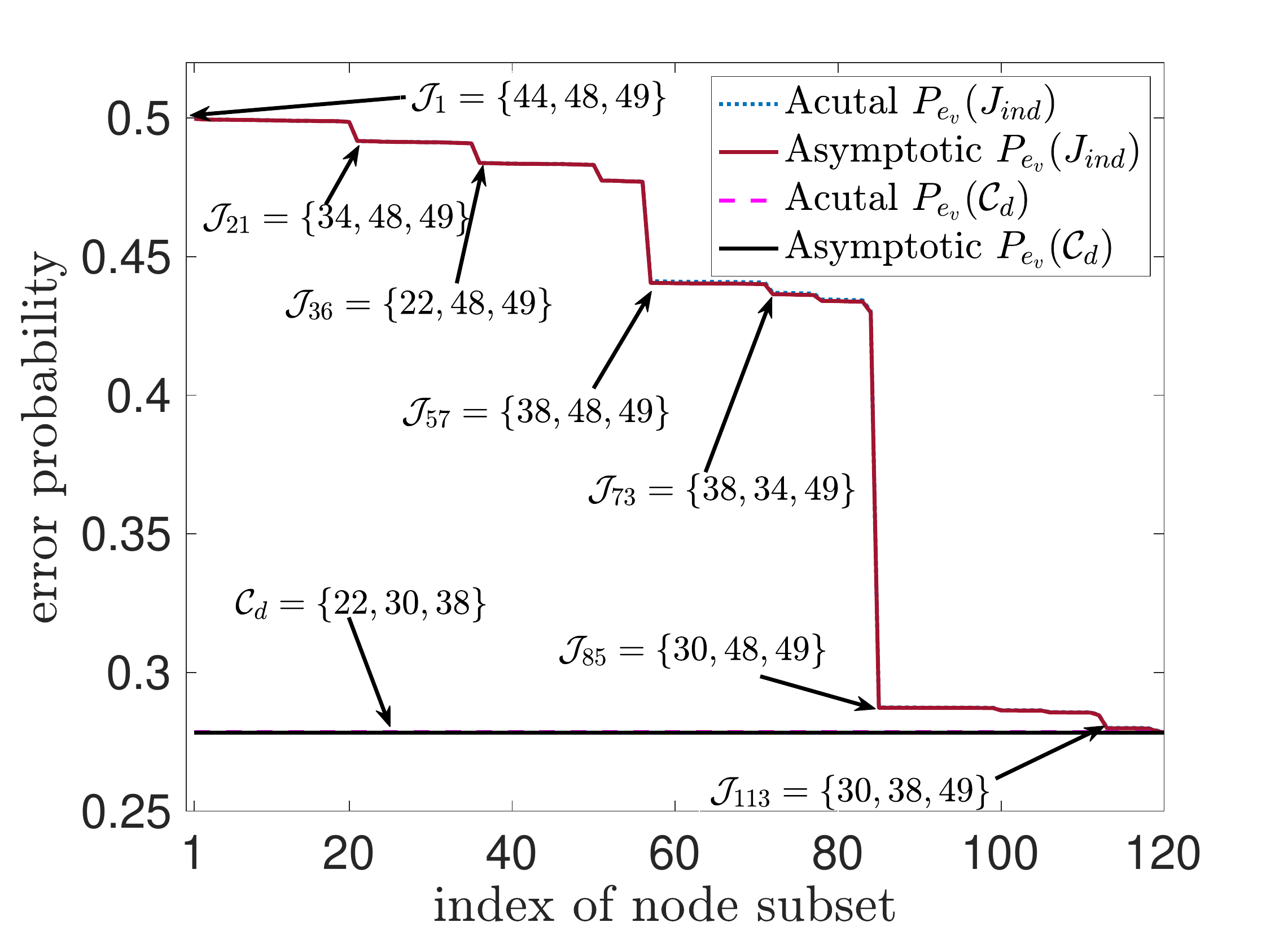}
		\caption{\footnotesize covariance shift model.}
		\label{fig: sensor_cov}
	\end{subfigure}
\caption{\footnotesize Actual and asymptotic error probabilities (Lemma \ref{lma: special cases error probability}) of the MAP and LD-MAP detectors associated with the node cutset $\mathcal{C}_d$ and all possible $3$ node subsets of $\mathcal{C}_d\sqcup \cal P$ of the network shown in Fig. \ref{fig: sensor_network}.The error probabilities of the detectors associated with cutset nodes is lower than that of the detectors associated with any subset of the nodes in the partitioned set. This result is consistent with Lemma \ref{lma: non-negative matrix}, because the submatrix $\widetilde G$ row sums of the adjacency matrix $G$ are less than $1/\sqrt{m_1}$ ($m_1=7$).}
\end{figure}

\section{Conclusion}\label{sec: conclusions} {In this
  paper we formulate mean and covariance detection placement problems
  for linear dynamical systems defined over networks with unknown
  stochastic inputs. The main technical contribution of the paper is
  to identify graphical conditions that predict the performance of MAP
  and LD-MAP detectors based on the distance between the employed
  sensors and the stochastic inputs. For networks with non-negative
  edge weights, we also show that the performance of a detector can be
  independent of the graphical distance between the sensors and the
  stochastic inputs.}

  % characterized algebraic relations between the network weights, that
  % illustrates the network configurations for which sensors far away
  % from the input outperform the sensors the input for detection
  % purposes. }

\setcounter{section}{1} \renewcommand{\thetheorem}{A.\arabic{theorem}}
\section*{APPENDIX}

\noindent {\it Proof of proposition \ref{prop: moments of measurements}}: 
From the network dynamics \eqref{eq: system} and sensor measurements \eqref{eq: measurement model}, ${\bf Y}_{\cal J}$ \eqref{eq: measurements recorded} can be expanded as 
\begin{align}\label{eq: expanded_measurements}
{\bf Y}_{\cal J}=\mathcal{O}{\bf x}[0]+\mathcal{F}{\bf w}_{0:N-1}+{\bf v}_{1:N}, 
\end{align}
where the vectors ${\bf w}_{0:N-1}=[{\bf w}[0]^\transpose, \ldots, {\bf w}[N-1]^\transpose]^\transpose$ and ${\bf v}_{1:N}=[{\bf v}[1]^\transpose, \ldots, {\bf v}[N]^\transpose]^\transpose$, respectively. The matrices $\mathcal{O}$ and $\mathcal{F}$ are defined in the statement of the proposition. The expressions of $\ol{\bs \mu}_i$ and ${\ol \Sigma}_i$ in \eqref{eq: moments} follows by taking the expectation and covariance of ${\bf Y}_{\cal J}$, respectively. \QEDB

\smallskip 
\noindent {\it Proof of Lemma \ref{lemma: MAP detector}}: Let $\zeta$ and $z$ are the realizations of ${\bf Y}_{\cal J}$ and $y$, respectively. Since the input and measurement noises follows a Gaussian distribution, the probability density functions of ${\bf Y}_{\cal J}$ \eqref{eq: measurements recorded} and $y={\bf b}^\transpose {\bf Y}_{\cal J}$ are  
\begin{align}\label{eq: pdfs}
f(\zeta|H_i)&\propto \frac{1}{\sqrt{|\overline \Sigma_i|}}\exp \left[ -\frac{1}{2}(\zeta-\ol{\bs \mu}_i)^{T}\overline{\Sigma}_i^{-1}(\zeta-\ol{\bs \mu}_{i})\right] \text{ and }\nonumber \\
g(z|H_i)&\propto \frac{1}{\sqrt{{\bf b}^\transpose \ol \Sigma_i {\bf b}}}\exp \left[ -\frac{(z-{\bf b}^\transpose \ol{\bs \mu}_i)^2}{2\,{\bf b}^\transpose \ol \Sigma_i {\bf b}}\right], 
\end{align}
respectively, where $|\cdot|$ denotes the determinant. Define the log likelihood ratios $\Psi(\zeta)=\ln(f(\zeta|H_2)/f(\zeta|H_1))$ and $\widehat\Psi(z)=\ln(f(z|H_2)/f(z|H_1))$. Then, from the mixed Bayes formula \cite{JM-DC:78}, the MAP decision rules based on $\zeta$ and $z$, respectively, are given by 
\begin{align}\label{eq: MAP pdfs}
\Psi(\zeta)\overset{\widehat{H}=H_{2}}{\underset{\widehat{H}=H_{1}}{\gtrless}}\gamma \text{ and } \widehat \Psi(z)\overset{\widehat{H}=H_{2}}{\underset{\widehat{H}=H_{1}}{\gtrless}}\gamma.  
\end{align}

{\bf part 1)} Since $ {\Sigma}_1={\Sigma}_2$ and $\bs{\mu}_1\ne \bs{\mu}_2$, from \eqref{eq: moments}, it follows that $\ol {\Sigma}_1=\ol {\Sigma}_2$ and $\ol{\bs \mu}_1\ne \ol{\bs \mu}_2$. Invoking this observation in $f(\zeta|H_i)$, yields the following expression for $\psi(\zeta)$: 
\begin{align}\label{eq: llr mean detector}
\Psi(\zeta)=-0.5\,\ol{\bs \mu}_\Delta^\transpose{\ol \Sigma}_2^{-1}\ol{\bs \mu}_\Delta+\left(y-\ol{\bs \mu}_1\right)^\transpose{\ol \Sigma}_2^{-1}\ol{\bs \mu}_\Delta. 
\end{align}
Substitute \eqref{eq: llr mean detector} in the first decision rule of \eqref{eq: MAP pdfs} and simplify the resulting expression to obtain the MAP decision rule \eqref{eq: MAP mean detector} for $\zeta$. Finally, replacing $\zeta$ with ${\bf Y}_{\cal J}$ yields the required expression. 

{\bf part 2)} In this case we have $\ol{\bs \mu}_1= \ol{\bs \mu}_2$ and $\ol {\Sigma}_1\ne \ol {\Sigma}_2$.  A similar procedure, as in part 1), based on $g(z|H_i)$ \eqref{eq: MAP pdfs} and the second decision rule in \eqref{eq: pdfs}, yields the LD-MAP detector's expression \eqref{eq: MAP mean detector}.  Details are left to the reader.  \QEDB

\smallskip 
\noindent {\it Proof of Lemma \ref{lma: special cases error probability}}: We divide the proof into two parts. In part 1) we derive the expressions \eqref{eq: P_em} and \eqref{eq: eta asymp} Instead, in part 2) we derive the expressions \eqref{eq: P_ev} and \eqref{eq: Rb asymp}. 

\smallskip 
{\bf part 1)} Under the assumption that $N<\infty$, let $\widehat \mbP_{e_m}(\cal J)$ be the error probability of \eqref{eq: MAP mean detector}.. Then, from \eqref{eq: MAP mean detector}, we have
\begin{align*}
\text{Pr}\left(\widehat H=H_2|H_1\right)&=\text{Pr}\left(\vphantom{\left. > \ol{\bs \mu}_\Delta^\transpose{\ol \Sigma}_c^{-1}\left( \ol{\bs \mu}_1+\ol{\bs \mu}_2\right)|H_1\right)}s> \ol{\bs \mu}_\Delta^\transpose{\ol \Sigma}_c^{-1}\left( \ol{\bs \mu}_1+\ol{\bs \mu}_2\right)|H_1\right) \text{ and}\\
\text{Pr}\left(\widehat H=H_1|H_2\right)&=\text{Pr}\left(\vphantom{\left. < \ol{\bs \mu}_\Delta^\transpose{\ol \Sigma}_c^{-1}\left( \ol{\bs \mu}_1+\ol{\bs \mu}_2\right)|H_1\right)}s< \ol{\bs \mu}_\Delta^\transpose{\ol \Sigma}_c^{-1}\left( \ol{\bs \mu}_1+\ol{\bs \mu}_2\right)|H_2\right),
\end{align*}
where $s=2\, \ol{\bs \mu}_\Delta^\transpose {\overline \Sigma}_c^{-1}  {\bf Y}_{\cal J}$ follows $\mathcal{N}(\ol{\bs \mu}_\Delta^\transpose{\ol \Sigma}_c^{-1}\ol{\bs \mu}_1, 4\ol{\bs \mu}_\Delta^\transpose{\ol \Sigma}_c^{-1}\ol{\bs \mu}_\Delta)$ under $H_i$, because $s$ is a linear transform of ${\bf Y}_{\cal J}|H_i$, which follows a Gaussian distribution. Define $\widehat \eta^2=\ol{\bs \mu}_\Delta^\transpose{\ol \Sigma}_c^{-1}\ol{\bs \mu}_\Delta$, and notice that $\text{Pr}(\widehat H=H_2|H_1)= Q_{\mathcal{N}}\left(0.5\,\widehat \eta\right)$ and $\text{Pr}(\widehat H=H_1|H_2)=1- Q_{\mathcal{N}}\left(0.5\,\widehat \eta\right)$. Finally, from \eqref{eq: general error probability}, we have $\widehat \mbP_{e_m}({\cal J})=0.5\,Q_{\cal N}(\widehat \eta)$. Define $\mbP_{e_m}({\cal J})= \lim_{N \to \infty}\widehat \mbP_{e_m}({\cal J})$, and note the following: 
\begin{align*}
\mbP_{e_m}({\cal J})=\lim\limits_{N \to \infty} 0.5\,Q_{\mathcal{N}}\left(0.5\,\widehat \eta\right)=0.5\,Q_{\mathcal{N}}\left(0.5\lim\limits_{N \to \infty} \widehat\eta\right)
%	&\overset{(b)}{=}0.5\,Q_{\mathcal{N}}\left(0.5 \, \eta \right), 
\end{align*}
where the final equality follows because $\widehat \eta$ is increasing in $N$ (see Proposition \ref{prop: eta and R versus N}). We now show that $\lim_{N \to \infty} \widehat\eta=\eta$. From \eqref{eq: moments}, it follows that
\begin{align}\label{eq: eta_proof}
\widehat \eta^2&=\left( \mathcal{F}{\bf m}\right)^\transpose {\ol \Sigma_c^{-1}} \left( \mathcal{F}{\bf m}\right), 
\end{align}
where  ${\bf m}={\bs 1}_N\otimes {\bs \mu}_{\Delta}$ and ${\bs \mu}_{\Delta}={\bs \mu}_2-{\bs \mu}_1$. Let $l=1,2,\ldots.$, and define $K(l)=CG^l\Pi$ and $S(i)=\sum_{l=0}^{i-1}K(l)$. With these definitions and the assumption $\overline{\lambda}(G)<1$, we have $\lim_{i \to \infty}S(i)=C(I-G)^{-1}\Pi\triangleq \overline K$, and 
\begin{align}\label{eq: mFJm_decomp_proof}
\mFJ{\bf m}&=\underbrace{\begin{bmatrix}
	S(1)-\overline K\\
	S(2)-\overline K\\
	\vdots \\
	S(N)-\overline K
	\end{bmatrix}}_{S_N} {\bs \mu}_\Delta+\left[ {\bs 1}_N\otimes \overline K\right] {\bs \mu}_\Delta. 
\end{align}
Let $t(S_N)={\bs \mu}_\Delta^\transpose \left[S_N^\transpose  \ol \Sigma_c^{-1} S_N+2S_N^\transpose\Sigma_c^{-1}  \left[{\bs 1}_N\otimes  \overline K\right]\right]  {\bs \mu}_\Delta$. By 
substituting \eqref{eq: mFJm_decomp_proof} in \eqref{eq: eta_proof}, we have
\begin{align}\label{eq: eta_square_oN}
\widehat \eta^2&={\bs \mu}_\Delta^\transpose \underbrace{\left[{\bs 1}_N\otimes  \overline K\right]^\transpose \ol \Sigma_c^{-1} \left[{\bs 1}_N\otimes  \overline K\right]}_F{\bs \mu}_\Delta+t(S_N).
\end{align}
Consider the first term of \eqref{eq: eta_square_oN}. Since ${\bf x}[0]=0$, from \eqref{eq: moments}, it follows that $\ol \Sigma_c\!=\!\left[\mathcal{F}\left(I_N\otimes \Sigma_c\right)\mathcal{F}^\transpose\!+\!\sigma^2_vI\right]$. Further,
\begin{align}\label{eq: M_form_two}
\left[{\bs 1}_N\otimes \overline K\right]^\transpose\ol \Sigma_c&=\left[ \overline K^\transpose\overline K\Sigma_c+\sigma^2_vI\right]\left[{\bs 1}_N\otimes \overline K\right]^\transpose+\nonumber \\ 
&\quad \underbrace{\overline K^\transpose\left[ \widetilde S_N^\transpose\left(I \otimes \Sigma_c\right)\mathcal{F}^\transpose+\overline K\Sigma_cS_N^\transpose\right]}_{\widetilde M}, 
\end{align}
where $\widetilde S_N$ is obtained by permuting, bottom to top, the block matrices of $S_N$ \eqref{eq: mFJm_decomp_proof}. Right multiplying either sides of \eqref{eq: M_form_two} with ${\ol \Sigma_c}^{-1}\left[{\bs 1}_N\otimes \overline K\right]$ gives us: 
\begin{align}\label{eq: M_right_multiplied}
N\overline K^\transpose\overline K\!=\!\left[\overline K^\transpose\overline K\Sigma_c+\sigma^2_vI\right]F +P,
\end{align}
where $P=\widetilde{M}{\ol \Sigma_c}^{-1}[{\bs 1}_N\otimes \overline K]$ and $F$ is defined in \eqref{eq: eta_square_oN}. Since $\overline K^\transpose\overline K\Sigma_c\succeq 0$, from \eqref{eq: M_right_multiplied}, it follows that $F=[\overline K^\transpose\overline K\Sigma_c+\sigma^2_vI]^{-1}[ N\overline K^\transpose\overline K-P]$. Substituting $F$ in \eqref{eq: eta_square_oN} yields 
\begin{align*}
\widehat \eta^2&=N{\bs \mu}_{\Delta}^\transpose([\overline K^\transpose\overline K\Sigma_c+\sigma^2_vI]^{-1}\overline K^\transpose \overline K){\bs \mu}_{\Delta}+\epsilon(N),  
\end{align*}
where $\epsilon(N)=-{\bs \mu}_{\Delta}^\transpose  [\overline K^\transpose \overline K\Sigma_c+\sigma^2_vI]^{-1}P {\bs \mu}_{\Delta}+t(S_N)$. Finally, substituting $K=L\Sigma_c^{-\frac{1}{2}}$ in the above expression,  and manipulating the terms will give us 
\begin{align}\label{eq: eta_J remainder}
\widehat \eta^2&=N\widetilde {\bs \mu}_{\Delta}^\transpose\left(  [ L^\transpose L+\sigma^2_vI] ^{-1}L^\transpose L\right)\widetilde{ \bs \mu}_{\Delta} +\epsilon(N). 
\end{align}

We claim that $\lim_{N\to\infty}\epsilon(N)=0$. To see this, rewrite $\epsilon(N)$ as ${\bs \mu}_{\Delta}^\transpose (Q_1(N)+Q_2(N)+Q_3(N)){\bs \mu}_{\Delta}$, where 
\begin{align*}
Q_1(N)&=S_N^\transpose\left[\ol \Sigma_c^{-1}S_N+2\Sigma_c^{-1}  \left[{\bs 1}_N\otimes  \overline K\right]\right]\\
Q_2(N)&=[\overline K^\transpose \overline K\Sigma_c+\sigma^2_vI]^{-1}\overline K^\transpose \widetilde S_N^\transpose\left(I \otimes \Sigma_c\right)\mathcal{F}^\transpose\\
Q_3(N)&=[\overline K^\transpose \overline K\Sigma_c+\sigma^2_vI]^{-1}\overline K^\transpose \overline K\Sigma_cS_N^\transpose.
\end{align*}
From part 1) of Assumption \ref{assump: error probability assumption}, there exist a $k \in \mathbb{N}$ such that for all $m \in \{k,k+1,\ldots,N\}$, $S(m)-\overline K=0$. Thus, all but finite rows of $S_N$ \eqref{eq: mFJm_decomp_proof} are zeros, i.e., we can express $S_N^\transpose$ as $[F_1^\transpose \,\, 0^\transpose]$ and $\widetilde S_N^\transpose$ as $[0^\transpose \,\, F_2^\transpose]$, where the dimension of $F_1$ and $F_2$ depends only $k$. Thus, for all $N>k$, $Q_i(N)$ is a constant matrix, say $Q_i$, and we may conclude that 
\begin{align*}
\|{\bs \mu}_{\Delta}\|^2_2\sum_{i=1}^3\lambda_\text{min}(Q_i+Q_i^\transpose)& \leq 2\sum_{i=1}^3{\bs \mu}_{\Delta}^\transpose Q_i{\bs \mu}_{\Delta} \\
&\leq  \|{\bs \mu}_{\Delta}\|^2_2\sum_{i=1}^3\lambda_\text{max}(Q_i+Q_i^\transpose), 
\end{align*}
where $\lambda_\text{max}(\cdot)$ and $\lambda_\text{min}(\cdot)$ are the maximum and minimum eigenvalues. Since $\lim_{N\to \infty}N \|{\bs \mu}_{\Delta}\|_2=c$ (Assumption \ref{assump: error probability assumption}), it follows that $\lim_{N\to \infty} \|{\bs \mu}_{\Delta}\|_2=0$. Hence, 
$\lim_{N\to \infty}\epsilon(N)=0$ and $\lim_{N\to \infty} \widehat \eta=\eta$ \eqref{eq: eta asymp}. 

\smallskip 
{\bf part 2)} Under the assumption that $N<\infty$, let $\widehat \mbP_{e_v}(\cal J)$ be the error probability of \eqref{eq: MAP covariance detector}. Then, from \eqref{eq: MAP mean detector}, we have
\begin{align*}
\text{Pr}\left(\widehat H=H_2|H_1\right)\!&=\!\text{Pr}\left(\ln(\widehat R)\!>\! \left[\frac{Z^2}{{\bf b}^\transpose \ol \Sigma_2 {\bf b}}-\frac{Z^2}{{\bf b}^\transpose \ol \Sigma_1 {\bf b}}\right]|H_1\right), \\
\text{Pr}\left(\widehat H=H_1|H_2\right)\!&=\!\text{Pr}\left(\ln(\widehat R)\!<\! \left[\frac{Z^2}{{\bf b}^\transpose \ol \Sigma_2 {\bf b}}-\frac{Z^2}{{\bf b}^\transpose \ol \Sigma_1 {\bf b}}\right]|H_2\right), 
\end{align*}
where $Z={\bf b}^\transpose[{\bf Y}_J-\ol{\bs \mu}_c]$ and $\widehat R=({\bf b}^\transpose \ol \Sigma_1 {\bf b}/({\bf b}^\transpose \ol \Sigma_2 {\bf b})>1$ (since $\Sigma_2=0$; Assumption \ref{assump: error probability assumption}). Let $U \sim \mathcal{N}(0,1)$. Then, $Z|H_i\overset{d}{=}(\sqrt{{\bf b}^\transpose \ol \Sigma_i {\bf b}}) U$, where $\overset{d}{=}$ means equality in the distribution. From this fact, we now have $\text{Pr}(\widehat H=H_2|H_1)=\text{Pr}\left(\widehat \tau > U^2\right)$ and  $\text{Pr}(\widehat H=H_1|H_2)=\text{Pr}(U^2> \widehat \tau \widehat R)$, where $\widehat \tau =\ln(\widehat R)/(\widehat R-1)$. Since $U^2\sim \chi^2(1)$, we finally have 
\begin{align*}
\widehat \mbP_{e_v}(\cal J)&\!=\!0.5\left[ 1-Q_{\chi^2}\left(1, \widehat \tau \right)\right] +0.5\, Q_{\chi^2}(1, \widehat \tau \widehat R). 
\end{align*}
To simplify $\widehat R$, note the following: since ${\bf b}$ is the maximizer of $I$-divergence \eqref{eq: I-divergence}, from \cite{LLS:91}, we can also express $\widehat R$ as 
\begin{align*}
\widehat R=\frac{{\bf b}^\transpose \ol \Sigma_1 {\bf b}}{{\bf b}^\transpose \ol \Sigma_2 {\bf b}}=\max\limits_{{\bf d}\, \in\, \mathbb{R}^{mN}}\, \frac{{\bf d}^\transpose \ol \Sigma_1 {\bf d}}{{\bf d}^\transpose \ol \Sigma_2 {\bf d}}. 
\end{align*}
Let ${\bf c}=\ol\Sigma^{1/2}_2{\bf d}$, and note the following: 
\begin{align*}
\widehat R&=\max\limits_{{\bf c}\, \in\, \mathbb{R}^{mN}}\left(\frac{\bf c}{\|{\bf c}\|_2}\right)^\transpose \ol \Sigma^{-1/2}_2\ol \Sigma_1 \ol \Sigma^{-1/2}_2\left(\frac{\bf c}{\|{\bf c}\|_2}\right)\\
&=\lambda_\text{max}\left(\Sigma^{-1/2}_2\ol \Sigma_1 \ol \Sigma^{-1/2}_2\right)=\lambda_\text{max}\left(\ol \Sigma_1 \ol \Sigma^{-1}_2\right).
\end{align*} 
Since $\widehat R$ is an increasing sequence, with respect to $N$ (see Proposition \ref{prop: eta and R versus N}), the limits $R=\lim_{N\to \infty}\widehat R$, $\tau=\lim_{N \to \infty}\widehat \tau$ and $\lim_{N \to \infty}\widehat \tau\widehat R=\tau R$ are well defined. Now, consider 
\begin{align*}
\mbP_{e_v}(\cal J)&=\lim_{N\to\infty}\widehat \mbP_{e_v}(\cal J)\\
&=\lim_{N\to\infty}0.5\left[ 1-Q_{\chi^2}\left(1, \widehat \tau \right)\right] +0.5\, Q_{\chi^2}(1, \widehat \tau \widehat R)\\
&=0.5\left[1-Q_{\chi^2}\left(1, \tau \right)\right]\!+\!0.5\, Q_{\chi^2}(1, \tau R), 
\end{align*}
where the last equality follows because $\widehat \tau$ and $\widehat \tau\widehat R$ are decreasing and increasing in $N$ (Proposition \ref{prop: eta and R versus N}), resp. 

We now show that $R$ is given by \eqref{eq: Rb asymp}. 
Since $\Sigma_2=0$ and ${\bf x}[0]=0$, we have $\ol\Sigma_2=\sigma^2_vI$ and $\ol \Sigma_1=FF^\transpose+\sigma^2_vI$, where $F=\mathcal{F}(I_N\otimes \Sigma_1^{\frac{1}{2}})$ and $\Sigma_1^{\frac{1}{2}}$ satisfies $\Sigma_1=\Sigma_1^{\frac{1}{2}}\Sigma_1^{\frac{1}{2}}$. From these observations, we may conclude that 
\begin{align}\label{eq: limit wide hat R}
R=\lim_{N\to \infty}\widehat R&=\lim_{N\to \infty}\frac{\lambda_\text{max}(FF^\transpose+\sigma^2_vI)}{\sigma^2_v}\nonumber\\
&=1+\sigma^{-2}_v\lim_{N\to \infty}\lambda_\text{max}(FF^\transpose). 
\end{align}
It now suffices to evaluate $\lim_{N\to \infty}\lambda_\text{max}(FF^\transpose$. Since $\overline{\lambda}(G)<1$, we may define the following matrix valued function \cite{RMG:06}: 
\begin{align*}
A(\omega)&=\sum^{\infty}_{l=0}K(l)\Sigma_1^{1/2}e^{jk\omega} \quad \omega \in [0,2\pi],
\end{align*}
where $K(l)=CG^l\Pi$ and $j=\sqrt{-1}$. Since the coefficients $K(l)\Sigma_1^{1/2}$ are absolutely summable, for any $l \in \mathbb{N}$, these coefficients can also be recovered as \cite{RMG:06}: 
\begin{align*}
K(l)\Sigma_1^{1/2}=\frac{1}{2\pi}\int_{0}^{2\pi}A(\omega)e^{-jl\omega}d\omega. 
\end{align*}
Let $\overline z$ be the conjugate of $z \in \mathbb{C}$. Then, from \cite[Chapter~6.4]{BA-SB:99}, we have 
\begin{align*}
\lim\limits_{N\to \infty}\lambda^{1/2}_\text{max}(FF^\transpose)&\!=\!\underset{\omega \in [0, 2\pi]}{\text{ess sup }}\|A(\omega)\|_2\\
&\!=\!\underset{\omega \in [0, 2\pi]}{\text{ess sup }}\left\Vert C \left( \sum^{\infty}_{l=0}G^le^{jl\omega}\right)\Pi \Sigma_1^{1/2}\right\Vert_2\\
&\!=\!\underset{\omega \in [0, 2\pi]}{\text{ess sup }}\left\Vert C \left(I-Ge^{jw}\right)^{-1}\Pi \Sigma_1^{1/2}\right\Vert_2\\
&\!=\!\underset{\{z\in \mathbb{C}:|z|=1\}}{\text{ess sup }}\left\Vert C \left(\overline z I-G\right)^{-1}\Pi \Sigma_1^{1/2}\right\Vert_2\\
&\!\overset{(a)}{=}\!\underset{\{z\in \mathbb{C}:|z|=1\}}{\text{ess sup }}\left\Vert C \left(z I-G\right)^{-1}\Pi \Sigma_1^{1/2}\right\Vert_2\\
&\!=\!\underset{\{z\in \mathbb{C}:|z|=1\}}{\text{ess sup }}\left\Vert T(z)\Sigma_1^{\frac{1}{2}}\right\Vert_2=||T(z)\Sigma_1^{\frac{1}{2}}||_{\infty}. 
\end{align*}
where $(a)$ follows because, for any $A \in \mathbb{C}^{N\times N}$ with $A^*$ denoting its complex conjugate transpose, $\|A\|_2=\|A^\transpose\|_2=\|A^*\|_2$. Substituting $ \lim_{N\to \infty}\lambda^{1/2}_\text{max}(FF^\transpose)$ in \eqref{eq: limit wide hat R} gives us $R=1+\sigma^{-2}_v||T^*(z)||^2_{\infty}$. \QEDB

\medskip 
\noindent{\it Proof of Theorem \ref{thm: network_theoretic_result_noiseless}}		
Let ${\bf y}_\calP[k]$ and ${\bf y}_\calS[k]$ denote the measurements of associated with the sensor sets $\cal P$ and $\cal C$, respectively. Since $\sigma^2_v=0$, from \eqref{eq: subsystem}, we have
\begin{align}\label{eq: functional_dynamics}
{\bf y}_{\cal P}[k+1]=G_{pp}{\bf y}_{\cal P}[k]+B{\bf y}_{\mathcal{C}}[k], 
\end{align}
where $B=G_{pc}$. From \eqref{eq: functional_dynamics}, it follows that 
\begin{align*}
\begin{split}
\underbrace{\begin{bmatrix}
	{\bf y}_{\cal P}[1]\\
	{\bf y}_{\cal P}[2]\\
	\vdots\\
	{\bf y}_{\cal P}[N]
	\end{bmatrix}}_{{\bf Y}_{\cal P}}&=\underbrace{\begin{bmatrix}
	G_{pp} & B\\
	G^2_{pp} & G_{pp}B\\
	\vdots & \vdots \\
	G^{N}_{pp} & G^{N-1}_{pp}B 
	\end{bmatrix}}_{M}\underbrace{\begin{bmatrix}
	{\bf y}_\calP[0]\\
	{\bf y}_\calC[0]
	\end{bmatrix}}_{\widehat{\bf Y}[0]}\\
&+\underbrace{\begin{bmatrix}
	0 & 0 & \cdots & 0 & 0\\
	B & 0 & \cdots & 0 &0\\
	\vdots & \vdots & \ddots & \vdots & \vdots\\
	G^{N-2}_{pp}B & G^{N-3}_{pp}B  & \cdots & B &0
	\end{bmatrix}}_{\widehat M}\underbrace{\begin{bmatrix}
	{\bf y}_{\mathcal{C}} [1]\\
	{\bf y}_{\mathcal{C}} [2]\\
	\vdots\\
	{\bf y}_{\mathcal{C}} [N]
	\end{bmatrix}}_{{\bf Y}_{\calC}}.
\end{split}
\end{align*}
Since $\widehat{\bf Y}[0]$ is independent of $H_i$, the assertion of the theorem follows from Lemma \ref{lma: dependent_measurements_poe}. \QEDB 

\medskip 
\noindent{\it Proof of Lemma \ref{lma: dependent_measurements_poe}} We shall prove the result assuming that ${{\bf Y}}$ and ${\bf Z}=g({\bf Y})+{\bf v}$ admits density functions. With the expense of notation, the given proof can be adapted to handle random variables that do not have densities. Let ${\bf L}=[{\bf Y}^\transpose, {\bf Z}^\transpose]^\transpose$. Consider the following log likelihood ratio (LR) based on ${\bf L}$: 
\begin{align*}
\frac{f(l|M_2)}{f(l|M_1)} &= \frac{f(y, g(y)+v|M_2)}{f(y, g(y)+v|M_1)}\\
& = \frac{f(y, g(y)+v|y,M_2)f(y|M_2)}{f(y, g(y)+v|y,M_1)f(y|M_1)}\\
& \overset{(a)}{=}\frac{f(y, g(y)+v|y)f(y|M_2)}{f(y, g(y)+v|y)f(y|M_1)}= \frac{f(y|M_2)}{f(y|M_1)}, 
\end{align*}
where (a) follows because ${\bf v}$ is independent of $M_i$. Since LRs of ${\bf L}$ and ${\bf Y}$ are equal, the error probabilities associated with their MAP rules should be the same. Instead, the error probability of the MAP rule based on ${\bf L}$ is always superior to that of ${\bf Y}$ or ${\bf Z}$ alone. Thus $\delta_1\leq \delta_2$.  \QEDB

\smallskip 
\noindent{\it Proof of Theorem \ref{thm: network_theoretic_result_noisy}} 
Consider the following deterministic analogue of \eqref{eq: system}: ${\bf x}[k+1]=G{\bf x}[k]+\Pi{\bf u}$, where ${\bf u}$ is arbitrary. Recall that
${\bf x}_p[k+1]=G_{pp}{\bf x}_p[k]+G_{pc}{\bf x}_c[k]$ \eqref{eq: subsystem}. Since ${\bf x}[0]=0$, for $z\notin \text{spec}(G)\cup \text{spec}(G_{pp})$, we have 
\begin{subequations}
	\begin{equation}
	{\bf x}[z]=(zI-G)^{-1}\Pi {\bf u}\text{ and }\label{eq: x[z]}
	\end{equation}
	\begin{equation}
	{\bf x}_p[z]=(zI-G_{pp})^{-1}G_{pc}{\bf x}_c[z]=T_s(z){\bf x}_c[z]\label{eq: x_p[z]}. 
	\end{equation}
\end{subequations}
From \eqref{eq: x_p[z]}, the following inequalities are obvious 
\begin{align}
\underline \rho(z) \|{\bf x}_c[z]\|_2\leq \|{\bf x}_p[z]\|_2\leq \overline \rho(z) \|{\bf x}_c[z]\|_2. \label{eq: x_p[z] inequality}
\end{align}
Let $C_1$ and $C_2$ be the sensor matrices associated with $\cal C$ and $\cal P$, respectively. Then,
\begin{align}\label{eq: xc xp relation with x}
{\bf x}_c[z]=C_1{\bf x}[z] \text{ and } {\bf x}_p[z]=C_1{\bf x}[z]. 
\end{align}

\smallskip 
{\bf part 1)} We now consider the cases 1a) and 1b). Let $L_i=C_i(I-G)^{-1}\Pi \Sigma_c^{\frac{1}{2}}$, where $\Sigma_c=\Sigma_c^{\frac{1}{2}}\Sigma_c^{\frac{1}{2}}$ is defined in Lemma \ref{lemma: MAP detector}. Let $z=1$. Then, from \eqref{eq: xc xp relation with x} note that 
\begin{align*}
\begin{split}
\|{\bf x}_c[1]\|_2^2&=\|C_1{\bf x}[1]\|_2^2={\bf u}^\transpose L_1^\transpose L_1{\bf u} \text{ and }\\
\|{\bf x}_p[1]\|_2^2&=\|C_2{\bf x}[1]\|_2^2={\bf u}^\transpose L_2^\transpose L_2{\bf u}. 
\end{split}
\end{align*}
From \eqref{eq: x_p[z] inequality} and above identities, it follows that 
\begin{align}\label{eq: pd inequalites}
\begin{split}
\overline \rho(1)<1  &\implies L_1^\transpose L_1+\sigma^2_vI\succ L_2^\transpose L_2+\sigma^2_vI \text{ and}\\
\underline \rho(1)>1 & \implies L_2^\transpose L_2+\sigma^2_vI\succ L_1^\transpose L_1+\sigma^2_vI. 
\end{split}
\end{align}
Let ${\bf u}=\widetilde{\bs \mu}_{\Delta}$, where $\widetilde{\bs \mu}_{\Delta}$ is defined in the statement of Lemma \ref{lma: special cases error probability}. Let $\eta_1$ and $\eta_2$ be the SNRs of $\mbP_{e_m}(\cal C)$ and $\mbP_{e_m}(\cal P)$, respectively. Then from \eqref{eq: eta asymp}, we have 
\begin{align*}
\eta^2_i&=N\widetilde {\bs \mu}_{\Delta}^\transpose\left(  [ L_i^\transpose L_i+\sigma^2_vI] ^{-1}L_i^\transpose L_i\right)\widetilde{ \bs \mu}_{\Delta}. 
\end{align*}
Using the identity $[ L_i^\transpose L_i+\sigma^2_vI]^{-1}L_i^\transpose L_i=I-\sigma^2_v[ L_i^\transpose L_i+\sigma^2_vI]^{-1}$, we can also express $\eta^2_i$ as 
\begin{align}\label{eq: eta_i}
\eta^2_i=\widetilde {\bs \mu}_{\Delta}^\transpose \widetilde {\bs \mu}_{\Delta}-
\sigma^2_v\,\widetilde {\bs \mu}_{\Delta}^\transpose\left[L_i^\transpose L_i+\sigma^2I\right]^{-1}\widetilde {\bs \mu}_{\Delta}^\transpose. 
\end{align}
Finally, from \eqref{eq: eta_i} and \eqref{eq: pd inequalites}, and Proposition \ref{prop: error probability vs eta and R}, we have
\begin{align*}
\overline \rho(1)&<1 \implies \eta^2_1\geq \eta^2_2\implies \mbP_{e_m}({\cal C}_d)\!\leq\!\mbP_{e_m}({\cal P}) \text{ and}\\
\underline \rho(1)&>1 \implies \eta^2_1\leq \eta^2_2\implies \mbP_{e_m}({\cal C}_d)\!\geq\!\mbP_{e_m}({\cal P}). 
\end{align*}

\smallskip 
{\bf part 2)} We now consider the cases 2a) and 2b). Let $T_i(z)=C_i(zI-G)^{-1}$. Let ${\bf u}=\Sigma_1^{1/2}{\bf d}$, where $\Sigma_1^{1/2}$ is defined in the statement of Lemma \ref{lma: special cases error probability}. From \eqref{eq: xc xp relation with x} and \eqref{eq: x[z]}, we have ${\bf x}_c[k]=T_1(z)\Sigma_1^{1/2}{\bf d}$ and ${\bf x}_c[k]=T_2(z)\Sigma_1^{1/2}{\bf d}$. By invoking these two facts in \eqref{eq: x_p[z] inequality}, we may now conclude that 
\begin{align*}
\sup_{|z|=1}\overline\rho(z)<1 &\implies \|T_2(z)\Sigma_1^{\frac{1}{2}}{\bf d}\|_2\leq \|T_1(z)\Sigma_1^{\frac{1}{2}}{\bf d}\|_2 \text{ and }\\
\inf_{|z|=1}\underline\rho(z)>1 &\implies \|T_2(z)\Sigma_1^{\frac{1}{2}}{\bf d}\|_2\geq \|T_1(z)\Sigma_1^{\frac{1}{2}}{\bf d}\|_2, 
\end{align*}
for all $z$ that satisfies $|z|=1$. Let $R_1$ and $R_2$ be the SNRs of $\mbP_{e_v}(\cal C)$ and $\mbP_{e_v}(\cal P)$, respectively. Then, from \eqref{eq: Rb asymp} 
\begin{align*}
R_i-1=\frac{\|T_i(z)\Sigma_1^{\frac{1}{2}}\|^2_\infty}{\sigma^2_v}=\left( \underset{\{z\in \mathbb{C}:|z|=1\}}{\text{ess sup }}\| T_2(z)\Sigma_1^{1/2}{\bf d}\|_2\right)^2. 
\end{align*}
From Proposition \ref{prop: error probability vs eta and R}, it follows that 
\begin{align*}
\sup_{|z|=1}\overline\rho(z)<1 & \implies R_1\geq R_2\implies \mbP_{e_v}({\cal C}_d)\!\leq\!\mbP_{e_v}({\cal P}) \text{ and}\\
\inf_{|z|=1}\underline\rho(z)>1 & \implies R_1\leq R_2\implies \mbP_{e_v}({\cal C}_d)\!\geq\!\mbP_{e_v}({\cal P}).  \quad \QEDB 
\end{align*}

\smallskip 
\noindent{\it Proof of Corollary \ref{lma: non-negative matrix}} We shall prove part 1) of the corollary, and part 2) can be derived using similar analysis (the details are omitted). The idea of the proof is to show that $||\widetilde G||_\infty\leq 1/\sqrt{m} \implies \overline \rho(1)<1$, and there upon invoking Theorem \ref{thm: network_theoretic_result_noisy} yields the desired assertion. 

\smallskip 
{\bf step 1)} For $G\geq 0$, it follows that $\sup_{|z|=1}\overline\rho(z)=\overline \rho(1)$, where $\overline \rho(z)$ is $\|(zI-G_{pp})^{-1}G_{pc}\|_2$. To see this, note the following: For any ${\bf d} \in {\mathbb C}^{n_1}$, let $|{\bf d}|=(|d_1|, \ldots, |d_{n_1}|)^\transpose$. Then, for any $l \in {\mathbb N}$ and $z$ that satisfies $|z|=1$, we have 
\begin{align*}
|(\overline{z}G_{pp})^lG_{pc}{\bf d}|=|(G_{pp})^lG_{pc}{\bf d}|\leq (G_{pp})^lG_{pc}|{\bf d}|, 
\end{align*}
where the inequality, to be understood coordinate wise, follows because $[G_{pp} \, G_{pc}]\geq 0$. From the above inequality, and the fact $|{\bf y}+{\bf z}|\leq |{\bf y}|+|{\bf z}|$ for any ${\bf x}, {\bf y} \in {\mathbb{C}}^{p}$, we have 
\begin{align*}
\left\vert\sum_{l=0}^{\infty}(\overline{z}G_{pp})^lG_{pc}{\bf d}\right\vert\leq \sum_{l=0}^{\infty}|(\overline{z}G_{pp})^lG_{pc}{\bf d}|\leq \sum_{l=0}^{\infty}(G_{pp})^lG_{pc}|{\bf d}|. 
\end{align*}
Since $G_{pp}$ is a submatrix of $G\geq$, which is a non-negative matrix, we have $|\lambda_\text{max}(\overline zG_{pp})|=|\lambda_\text{max}(G_{pp})|\leq |\lambda_\text{max}(G)|<1$ \cite{AB-PRJ:94}, the above inequality can also be expressed as 
\begin{align*}
\left\vert(I-\overline z G_{pp})^{-1}G_{pc}{\bf d}\right\vert\leq (I-G_{pp})^{-1}G_{pc}\left\vert{\bf d}\right\vert.
\end{align*}
Taking $2$-norm on both sides of the inequality yields us: 
\begin{align*}
\left\Vert \, \left\vert(I-\overline z G_{pp})^{-1}G_{pc}{\bf d}\right\vert\, \right\Vert_2 \leq \left\Vert \, (I-G_{pp})^{-1}G_{pc}\left\vert{\bf d}\right\vert\, \right\Vert_2. 
\end{align*}
Since the above inequality holds for any vector ${\bf d} \in {\bf R}^{n_1}$, using the identity $\|\,|{\bf x}|\,\|_2=\|\,{\bf x}\,\|_2$ for any ${\bf x} \in {\mathbb{C}}^{p}$, the following inequality is now obvious: 
\begin{align*}
\sup_{|z|=1}\sup_{\|{\bf d}\|_2=1}\|(I-\overline z G_{pp})^{-1}G_{pc}{\bf d}\|_2\leq \left\Vert \, (I-G_{pp})^{-1}G_{pc} \, \right\Vert_2, 
\end{align*}
which can be expressed as $\sup_{|z|=1}\overline \rho(z) \leq \overline \rho(1)$. The equality is attained at $z=1$. 

\smallskip 
{\bf step 2)} Since $\sup_{|z|=1}\overline\rho(z)=\overline \rho(1)$, from Theorem \ref{thm: network_theoretic_result_noisy} it readily follows that, both 
$\mbP_{e_m}({\cal C}_d)\!\leq\!\mbP_{e_m}({\cal P})$ and $\mbP_{e_v}({\cal C}_d)\!\leq\!\mbP_{e_v}({\cal P})$ holds true whenever $\overline \rho(1)<1$. We now show that 
$\|\widetilde G\|_{\infty}=\|[G_{pp}\, G_{pc}]\|_{\infty}<1/\sqrt{m_1}$ guarantees $\rho(1)<1$. Let ${\bf 1}$ denote the all ones vector, and note the following identity: 
\begin{align}\label{eq: block matrix ones trick}
\begin{bmatrix}
G_{pp} & G_{pc}\\
0 & I
\end{bmatrix}^k\begin{bmatrix}
{\bf 1}\\
{\bf 1}
\end{bmatrix}=\begin{bmatrix}
G_{pp}^k & \sum_{l=0}^{k-1}G_{pp}^{l}G_{pc}\\
0 & I
\end{bmatrix}^k\begin{bmatrix}
{\bf 1}\\
{\bf 1}
\end{bmatrix}. 
\end{align}
\smallskip 
Since $\|[G_{pp}\,\, G_{pc}]\|_{\infty}<1/\sqrt{m_1}$, for any $k \in {\mathbb N}$, we also have  
\begin{align*}
\begin{bmatrix}
G_{pp} & G_{pc}\\
0 & I
\end{bmatrix}^k\begin{bmatrix}
{\bf 1}\\
{\bf 1}
\end{bmatrix}\leq \begin{bmatrix}
\frac{1}{\sqrt{m_1}}{\bf 1}\\
{\bf 1}
\end{bmatrix}. 
\end{align*}
From the above inequality and \eqref{eq: block matrix ones trick}, it follows that 
\begin{align*}
\lim\limits_{k \to \infty}\begin{bmatrix}
G_{pp}^k & \sum_{l=0}^{k-1}G_{pp}^{l}G_{pc}\\
0 & I
\end{bmatrix}^k\begin{bmatrix}
{\bf 1}\\
{\bf 1}
\end{bmatrix}\leq \begin{bmatrix}
\frac{1}{\sqrt{m_1}}{\bf 1}\\
{\bf 1}
\end{bmatrix}. 
\end{align*}
Since $\overline{\lambda}(G_pp)<1$, as $k \to \infty$, it follows that $G^k_{pp}\to 0$ and $ \sum_{l=0}^{k-1}G_{pp}^{l}G_{pc} \to (I-G_{pp})^{-1}G_{pc}$. Thus $(I-G_{pp})^{-1}G_{pc}{\bf 1}=\|(I-G_{pp})^{-1}G_{pc}\|_\infty<1/\sqrt{m_1}$, 
and hence, $\overline \rho(1)=||(I-G_{pp})^{-1}G_{pc}||_2<\sqrt{m_1}||(I-G_{pp})^{-1}G_{pc}||_{\infty}<1$. \QEDB 

\smallskip 
\noindent{\it Proof of Proposition \ref{prop: error probability vs eta and R}} Since $Q_{\cal N}(x)$ is decreasing function of $x$, $\mbP_{e_m}(\cal J)$ \eqref{eq: P_em} is decreasing in SNR $\eta$, given by either \eqref{eq: eta asymp} or  \eqref{eq: identical eta asymp}. For $\mbP_{e_v}(\cal J)$ \eqref{eq: P_ev} note the following: first, observe that $R>1$ in both \eqref{eq: Rb asymp} and \eqref{eq: identical R asymp}. Thus
\begin{align}\label{eq: tau R derivatives}
\begin{split}
\frac{d\tau}{dR}&=\frac{\left( \frac{R-1}{R}\right) -\ln R}{(R-1)^2}<0, \text{ and }\\
\frac{d(\tau R)}{dR}&=\frac{\left( R-1\right) -\ln R}{(R-1)^2}>0. 
\end{split}
\end{align}
Hence, we conclude that $\tau$ is deceasing in $R$. Instead, $\tau R$ is increasing in $R$. From this observation and the fact that the $Q_{\chi^2}(1,z)=\text{Pr}[Z\geq z]$, where $Z\sim \chi^2(1)$, is decreasing in $z$, it follows that $\mbP_{e_v}(\cal J)$ is decreasing in $R$. \QEDB  

\smallskip 
\noindent{\it Proof of Proposition \ref{prop: top monontone}} From \eqref{eq: Toeplitz adjacency matrix}, and the fact that $1\leq q<j<\ldots<n$, where $\mathcal{C}_d=\{j\}$ and $\mathcal{P}=\{j+1,\ldots,n\}$, the row sums of $\widetilde G$ takes values in the set $\{a+c, a+b+c\}$. Let $|\overline{G}_{l,q}|=|(I-G)^{-1}_{l,q}|$. Using the principle of backward induction, we shall show that, when $\|\widetilde G\|_{\infty}=a+b+c<1$, $\{|\ol G_{lq}|\}_{l=q}^n$ is monotonically decreasing. The proof of part (ii) is left to the reader as an exercise. 

Let $\ta=1-a$, $\tb=-b$, $\tc=-c$. If $\ta\ne 0$, then $\ol G_{l,q}$ of $(I-G)^{-1}$ are given by the following expressions \cite{JWL:82}: 
\begin{align}\label{eq: G bar elements}
\ol G_{l,q}&=\frac{1}{\theta_n}\left\{\begin{array}{lr}
(-1)^{l+q}\tb^{q-l}\theta_{l-1}\phi_{q+1} & q\geq l\\
(-1)^{l+q}\tc^{l-q}\theta_{q-1}\phi_{l+1} & q<l
\end{array}\right.
\end{align}
where $l,q\in \{1,\ldots,n\}$, and $\theta_k$ and $\phi_k$ are governed by
\begin{align}\label{eq: toeplitz recursions}
\begin{split}
\theta_k&=\ta\theta_{k-1}-\tb\tc\theta_{k-2} \quad \text{for } k= 2,\ldots,n\\
\phi_k&=\ta\phi_{k+1}-\tb\tc\phi_{k+2} \quad \text{for } k= n-1,\ldots,1
\end{split}
\end{align}
where $\theta_0=1$, $\theta_1=\ta$, $\phi_n=\ta$, $\phi_{n+1}=1$ and $\theta_n=\text{det}(I-G)$. Let $\mathcal{L}=\{q+1,\ldots,n\}$. Then, for any $l \in \mathcal{L}\cup \{q\}$, 
\begin{align*}
|\ol G_{lq}|\triangleq |(I-G)^{-1}_{lq}|&=\left\vert\frac{\theta_{q-1}\widetilde c^{-q}}{\theta_n}\right\vert \left\vert\widetilde c^l\phi_{l+1}\right\vert, 
\end{align*}
Let $l \in \mathcal{L}$, and define  $\zeta(l)=|\ol G_{l,q}|/|\ol G_{l-1,q}|$. Since $\phi_{n+1}=1$ and $\phi_{n}=\widetilde a$, for $l=n$ (base step), it follows that 
\begin{align*}
\zeta(n)=\frac{|\ol G_{nq}|}{|\ol G_{n-1,q}|}=\frac{\left\vert\widetilde c^n\phi_{n+1}\right\vert}{\left\vert\widetilde c^{n-1}\phi_{n}\right\vert}=\frac{|\widetilde c|}{|\widetilde a|}=\frac{c}{1-a}\overset{(i)}{<}1, 
\end{align*}
where $(i)$ follows because $a,b,c>0$, and $a+b+c<1$. Let $q<l<n$ and $\zeta(l+1)<1$ (inductive step). Then, 
\begin{align*}
\zeta(l)=\frac{|\ol G_{l,q}|}{|\ol G_{l-1,q}|}&=\frac{|\widetilde c|\left\vert\phi_{l+1}\right\vert}{\left\vert\phi_{l}\right\vert}\overset{\eqref{eq: toeplitz recursions}}{=}\frac{c}{\left\vert \widetilde a-\widetilde b \widetilde c\left( \frac{\phi_{l+2}}{\phi_{l+1}}\right) \right\vert}<1, 
\end{align*}
To see the last inequality, consider the following: 
\begin{align*}
b+c < 1-a &\overset{(ii)}{\implies} b\left(\frac{|\widetilde c| |\phi_{l+2}|}{|\phi_{l+1}|}\right) +c<1-a \\
&\implies b\left(\frac{c \phi_{l+2}}{\phi_{l+1}}\right) +c<1-a \\
&\implies\frac{c}{\left\vert (1-a)-bc\left( \frac{\phi_{k+2}}{\phi_{k+1}}\right) \right\vert}<1\\
& \overset{(iii)}{\implies}\frac{|\widetilde c|}{\left\vert \widetilde a-\widetilde b \widetilde c\left( \frac{\phi_{l+2}}{\phi_{l+1}}\right) \right\vert}<1, 
\end{align*}
where $(ii)$ follows because the hypothesis $\zeta(l+1)<1$ implies that  $|\widetilde c| (|\phi_{l+2}|/|\phi_{l+1}|)<1$, and $(iii)$ from the fact that 
$\widetilde a=1-a$, $\widetilde b=-b$, and $\widetilde c=-c$. From the principle of finite induction, for all $l \in \mathcal{L}$, we have $\zeta(l)<1$. Hence, $\{|\ol G_{lq}|\}_{l=q}^n$ is a decreasing sequence. \QEDB 

\smallskip 
\begin{proposition}\label{prop: eta and R versus N}
	Let $\widehat \eta^2=\ol{\bs \mu}_\Delta^\transpose{\ol \Sigma}_c^{-1}\ol{\bs \mu}_\Delta$, $\widehat R=\lambda_\text{max}(\ol \Sigma_1 \ol \Sigma^{-1}_2)$ and $\widehat \tau =\ln(\widehat R)/(\widehat R-1)$, where ($\overline {\bs \mu}_{\Delta}$, $\ol\Sigma_c$, $\ol\Sigma_1$, $\ol\Sigma_2$) are defined in the statement of Lemma \ref{lemma: MAP detector}.
	Then, $\widehat \eta$, $\widehat R$, and $\tau$ are increasing in $N$. However, $\widehat \tau\widehat R$ is decreasing in $N$. 
\end{proposition}
\begin{proof}
	Let $N<\infty$. Then, from Proposition \ref{prop: moments of measurements}, we have $\ol{\bs \mu}_\Delta=\mathbb{E}\left[{\bf Y}_{\cal J}|H_2\right]-\mathbb{E}\left[{\bf Y}_{\cal J}|H_1\right]$, $\ol \Sigma_c=\text{Cov}[{\bf Y}_{\cal J}|H_1]=\text{Cov}[{\bf Y}_{\cal J}|H_2]$. For clarity, we drop the existing subscripts and replace them with the total number of measurements. Let $N_2=N_1+k$, $k \in \mathbb{N}$, and consider ${\bf Y}_{N_2}^\transpose=\begin{bmatrix}{\bf Y}_{N_1}^\transpose,\, {\bf Z}^\transpose_{k} \end{bmatrix}$, where ${\bf Z}_{k}$ are the measurements collected after $N_1$. Then, 
	\begin{align*}
	\ol{\bs \mu}_{N_2}\!=\!\renewcommand\arraystretch{1.1}
	\left[
	\begin{array}{c}
	\ol{\bs \mu}_{N_1}\\
	\hline
	{\bf m}_k
	\end{array}
	\right] \text{ and } \ol\Sigma_{N_2}\!=\!\renewcommand\arraystretch{1.1}
	\left[
	\begin{array}{c|c}
	\ol\Sigma_{N_1} & D \\
	\hline
	D^\transpose & M
	\end{array}
	\right],
	\end{align*} 
	where ${\bf m}_k=\mathbb{E}[{\bf Z}_k|H_2]-\mathbb{E}[{\bf Z}_k|H_1]$, $M=\text{Cov}[{\bf Z}_k|H_1]>0$,  and $D\!=\!\text{Cov}[{\bf Y}_k, {\bf Z}_k|H_1]$. Further, using the Schur complement, $\ol\Sigma_{N_2}^{-1}$ can be expressed as 
	\begin{align*}
	\ol\Sigma_{N_2}^{-1}\!=\!\renewcommand\arraystretch{1.1}
	\left[
	\begin{array}{c|c}
	\ol\Sigma_{N_1} & D \\
	\hline
	D^\transpose & M
	\end{array}
	\right]^{-1}=\left[
	\begin{array}{c|c}
	\ol\Sigma_{N_1}^{-1} & 0 \\
	\hline
	0^\transpose & 0
	\end{array}
	\right]+\underbrace{F}_{>0}. 
	\end{align*}
	From the above identity, it follows that
	\begin{align*}
	\widehat \eta_{N_2}=\left( \ol{\bs \mu}_{N_2}^\transpose\ol\Sigma^{-1}_{N_2}\ol{\bs \mu}_{N_2}\right)^{\frac{1}{2}}&=\left(\ol{\bs \mu}_{N_1}^\transpose\ol\Sigma^{-1}_{N_1}\ol{\bs \mu}_{N_1}+\ol{\bs \mu}_{N_2}^\transpose F \ol{\bs \mu}_{N_2}\right)^{\frac{1}{2}}\\
	& \geq \left( \ol{\bs \mu}_{N_1}^\transpose\ol\Sigma^{-1}_{N_1}\ol{\bs \mu}_{N_1}\right)=\widehat \eta_{N_1}. 
	\end{align*}
	Hence, we may conclude that $\widehat \eta$ is increasing in $N$. Instead, from the eigenvalue interlacing property for the symmetric matrix pencils \cite{R:15}, it follows that $\widehat R=\lambda_\text{max}(\ol \Sigma_1\ol\Sigma_2^{-1})$ is increasing in $N$. Finally, from \eqref{eq: tau R derivatives}, it follows that $\widehat \tau$ and $\widehat \tau \widehat R$ are decreasing and increasing in $N$, respectively. 
\end{proof}

\bibliographystyle{elsarticle-harv}
\bibliography{BIB}

\end{document}